\documentclass[12pt]{article}
\usepackage[utf8]{inputenc}
\usepackage{lmodern}
\usepackage{graphicx, xcolor, mathrsfs}
\usepackage{amsfonts}
\usepackage{blkarray, bigstrut}
\usepackage{mathtools}
\usepackage[nice]{nicefrac}
\usepackage{amsmath,amsthm,amssymb}
\usepackage[mathlines]{lineno}

\usepackage{enumerate}
\usepackage{enumitem}
\usepackage[colorlinks=true,citecolor=black,linkcolor=black,urlcolor=blue]{hyperref}
\usepackage[top=.9in, bottom=.9in, left=.5in , right=.5in]{geometry}
\usepackage{caption}
\usepackage{xifthen} 
\usepackage{wrapfig}
\usepackage[numbers, square]{natbib}
\usepackage{changes}
\usepackage{comment}
\usepackage{float}
\usepackage{ifthen}
\mathtoolsset{showonlyrefs}
\usepackage{tikz}
\usepackage{caption}                       
\usetikzlibrary{trees, arrows}
\usetikzlibrary{graphs,graphs.standard}
\usetikzlibrary{positioning,arrows.meta}
\usetikzlibrary{shapes.multipart}
\usetikzlibrary{arrows}
\usetikzlibrary{automata}
\definecolor{arrowblue}{RGB}{0,0,0}  

\usepackage[splitrule]{footmisc} 
\usepackage{caption}
\newtheorem{thm}{Theorem}[section]
\newtheorem*{thm*}{Theorem}
\newtheorem{lem}[thm]{Lemma}
\newtheorem{clm}{Claim}[thm]
\newtheorem{prop}[thm]{Proposition}
\newtheorem{cor}[thm]{Corollary}
\newtheorem{assumption}{Assumption}[section]

\theoremstyle{definition}
\newtheorem{rmq}[thm]{Remark}
\newtheorem{ex}[thm]{Example}
\newtheorem{exs}[thm]{Examples}
\newtheorem{defn}[thm]{Definition}

\def\Var{\mathop{\rm Var}\nolimits}

\newcommand{\N}{\mathbb N}
\newcommand{\E}[2][]{\ensuremath{\mathbb{E}_{#1}\left[#2 \right]}}
\newcommand{\Prob}[2][]{\ensuremath{\mathbb{P}_{#1} \left(#2 \right)}}
\newcommand{\var}[2][]{\ensuremath{\Var_{#1} \left(#2 \right)}}

\newcommand{\eps}{\varepsilon}

\DeclareMathOperator{\exponentialrv}{Exp}

\newcommand{\Exp}[1]{\exponentialrv\left( #1 \right)}
\DeclareMathOperator{\degg}{deg}

\DeclareMathOperator{\win}{Win}
\newcommand{\outdeg}[1]{\ensuremath{\degg^{+}(#1)}}

\newcommand{\ensymboldremark }{\hfill$\blacktriangleleft$}

\title{Persistent hubs in CMJ branching processes with independent increments and preferential attachment trees.}
\author{T. Iyer\footnote{Weierstrass Institute for Applied Analysis and Stochastics, Mohrenstrasse 39, 10117 Berlin, Germany.}}
\date{\today}

\begin{document}
\maketitle
\abstract{A sequence of trees $(\mathcal{T}_{n})_{n \in \mathbb{N}}$ contains a \emph{persistent hub}, or displays \emph{degree centrality}, if there is a fixed node of maximal degree for all sufficiently large $n \in \mathbb{N}$. We derive sufficient criteria for the emergence of a persistent hub in genealogical trees associated with Crump-Mode-Jagers branching processes with independent waiting times between births of individuals, and sufficient criteria for the non-emergence of a persistent hub.
We also derive criteria for uniqueness of these persistent hubs. As an application, we improve results in the literature concerning the emergence of unique persistent hubs in generalised preferential attachment trees, in particular, allowing for cases where there may not be a \emph{Malthusian parameter} associated with the process. The approach we use is mostly self-contained, and does not rely on prior results about Crump-Mode-Jagers branching processes. }
\section{Introduction}

In the study of models associated with the evolution of complex networks, it is of interest to understand the location of dominant `hubs' in the network. If one identifies these hubs as nodes of maximal degree, a natural question is to consider whether hubs appear `early' in the evolution of a network, or whether newer and newer hubs continue to arise. These hubs can have different meanings in different contexts. In a structure such as a social network a hub might indicate the presence of an `influencer'. On the other hand, in a structure representing genealogical trees associated with species of a virus, a hub might indicate a type of virus associated with a `super-spreader' (assuming that the number of mutant offspring produced by certain virus is positively correlated with its population).  

Often, in such structures, one expects a degree of reinforcement, or a `rich-gets-richer' effect. In the well-known preferential attachment model, introduced in the context of networks by Albert and Barab\'{a}si in \cite{barabasi}, new nodes attach to existing nodes according to their degree, so that nodes of large degree are reinforced. This model is widely studied because it displays features closely related to real-world networks. There is, by now, a large body of literature on such models, and variants. For a broader review we refer the reader to~\cite{Hof16}.

In the classical preferential attachment models, `older' nodes are likely to have large degree earlier, and are reinforced to such an extent that a single `old' node becomes the node of maximal degree throughout the evolution of the process (the precise asymptotics for this growth were derived in~\cite{mori-maximal}). In~\cite{dereich-morters-persistent} such a node is referred to as a persistent hub, whilst in other works, such as~\cite{banerjee-bhamidi-centrality}, when such a node exists one says there is `persistence of degree centrality' in the underlying model. 

A natural question is to consider the degree of reinforcement required to ensure that a persistent hub emerges. In \emph{generalised preferential attachment models} (for example,~\cite{rudas}), newer nodes instead connect to existing nodes with probability proportional to a positive function $f$ of the degree of that node. A conjecture motivated by works such as ~\cite{dereich-morters-persistent} is that, in general, a persistent hub emerges if and only if $\sum_{i=1}^{\infty} 1/f(i)^2 < \infty$. In~\cite{dereich-morters-persistent} the authors prove this conjecture in particular cases in a simplified variant of the preferential attachment model in which newly arriving nodes connect to a random number of existing nodes. 
In models more closely related to the model from~\cite{barabasi}, where the attachment step involves a possibly random denominator, this conjecture was proved in the case that $f$ is unbounded and convex by Galashin in~\cite{galashin}, and later extended to a wider degree of functions by Banerjee and Bhamidi in~\cite{banerjee-bhamidi}.  

There are a number of other ways of measuring `influence' or \emph{centrality} of nodes in a network than degree. In~\cite{banerjee-bhamidi-centrality}, building on previous works by Jog and Po-Ling-Lo~\cite{jog-loh-centrality-preferential-attachment, jog-loh-persistence-1},  Banerjee and Bhamidi showed, under certain technical conditions, that the condition~$\sum 1/f(i)^2 < \infty$ is a sufficient condition for persistence of many more general centrality measures that just `degree centrality'. Other works deal with preferential attachment type graphs with an inhomogeneous structure, where nodes are equipped with random weights that influence their evolution, so that nodes with larger weights are more likely to obtain new connections. In~\cite{bas} criteria are derived for a phase transition in the location of the node of `maximal degree' in a preferential attachment model with an additive weight. In particular, criteria are derived, based on the distribution of the random weight, for whether or not there emerges a persistent hub. The works in, for example,~\cite{eslava-high-deg, dereich-mailler-morters, LodOrt20, freeman-jordan, banerjee-bhamidi-change-point} deal with analysis of maximal degrees in other variants of the model. 

However, a natural limitation arises in the above models:

\begin{itemize}
    \item Rather than evolving in discrete time-steps, many real-world complex networks change continuously over time, and
    \item When evolving continuously over time, it may not be the case that the process satisfies the Markov property. For example, it may be unrealistic to require that the `waiting time' for a node to acquire a new link is memoryless. 
\end{itemize}

A means of overcoming these limitations is to analyse trees associated with \emph{Crump-Mode-Jagers} (CMJ) branching processes. Although one can only directly study trees in this framework, rather than more general graphs, this may not always be undesirable in applications. For example, CMJ branching processes arise naturally in modelling the sizes of populations growing over time, such as infected populations during epidemics (see, e.g.,~\cite{Komjathy2016ExplosiveCB, ball-epidemic-2016, Levesque-21}). In such contexts, knowledge about the location of `hubs' in the tree may provide a better understanding about the locations of `super-spreaders' during an epidemic. 

Working with CMJ branching processes to analyse models of discrete random trees is not new. They have already been applied to the analysis of preferential attachment type models in, for example~\cite{rudas, bhamidi, holmgren-janson, rec-trees-fit}. In~\cite{vdh-aging-mult-fitness-2017}, a particular CMJ branching process is used to analyse the citation networks. Most related to this work with regards to techniques, however, are previous works that deal with genealogical trees of CMJ branching processes in `explosive regimes' in~\cite{inhom-sup-pref-attach, limiting-structure, Komjathy2016ExplosiveCB}. These results, whilst rather general, are limited in that that the trees associated with the models can only represent extreme phenomenon in networks. For example, when applicable to preferential attachment type trees, the limiting infinite trees associated with the models are either locally finite, or only have a single node of infinite degree~\cite[Theorem~2.12]{inhom-sup-pref-attach}. 

\subsection{Overview of our contributions and structure}
\begin{enumerate}
    \item Our most general results concern sequences of genealogical trees associated with CMJ branching processes with independent waiting times between births of individuals.
    In Theorem~\ref{thm:persistence}, we provide sufficient criteria for the almost sure emergence of a persistent hub and sufficient criteria for almost sure non-existence of a persistent hub in these sequences. In Theorem~\ref{thm:unique} we provide criteria under which, almost surely, a \emph{unique} persistent hub arises.
    \item In Theorem~\ref{thm:main}, Corollary~\ref{cor:pers-det-trees} we apply these results to generalised preferential attachment models. We provide criteria under which, almost surely, a unique persistent hub arises and criteria under which, almost surely, no persistent hub arises. 
    \item Another result, Theorem~\ref{thm:superlinear}, proves an additional sufficient criterion for a unique persistent hub in the generalised preferential attachment trees. The techniques used in this result are limited to generalised preferential attachment trees, and do not carry over to more general CMJ branching processes. 
\end{enumerate}
Our results in Theorem~\ref{thm:persistence} and Theorem~\ref{thm:unique} are novel for CMJ branching processes. In particular, we do not require the common assumption of a \emph{Malthusian parameter}, which appears in many foundational results concerning CMJ branching processes (e.g.~\cite{nerman_81, nerman-jagers-84, jagers-nerman89, jagers-nerman-96, olofsson-x-log-x, iksanov2023asymptotic}). However, we conjecture that there are more general, necessary and sufficient criteria for the emergence of a persistent hub in these processes - see Remark~\ref{rem:general-conj}.
Our results in Theorem~\ref{thm:main}, Corollary~\ref{cor:pers-det-trees} and Theorem~\ref{thm:superlinear}, concerning generalised preferential attachment trees, improve results concerning preferential attachment trees by Banerjee and Bhamidi in~\cite{banerjee-bhamidi}, and Galashin~\cite{galashin} (see also Remark~\ref{rem:imp-1} and Examples~\ref{exs:imp-2}). 

The rest of this paper is structured as follows:
\begin{enumerate}
    \item Section~\ref{sec:statements} deals with a general description of the model and the main statements of results. In Definition~\ref{def:persistence} we define persistent hubs with regards to a sequence of directed trees. Then, 
    \begin{enumerate}
        \item Section~\ref{sec:cmj-results} includes a general description of CMJ branching processes (including relevant notation), and the statements of Theorem~\ref{thm:persistence} and Theorem~\ref{thm:unique}.
        \item Section~\ref{sec:preferential} includes a general description of the generalised preferential attachment model, and the statements of Theorem~\ref{thm:main}, Corollary~\ref{cor:pers-det-trees} and Theorem~\ref{thm:superlinear}. 
    \end{enumerate}
    For readers interested in statements of results, Section~\ref{sec:cmj-results} and Section~\ref{sec:preferential} may be read independently.
    \item Section~\ref{sec:proofs} then includes proofs of the results in this paper: Section~\ref{sec:auxiliary-lemmas} includes some auxiliary lemmata useful in the proofs that follow. Section~\ref{sec:cmj-proofs} includes the proof of Theorem~\ref{thm:persistence} and Theorem~\ref{thm:unique} and Section~\ref{sec:pref-attach-proofs} includes the proofs of Theorem~\ref{thm:main} and Theorem~\ref{thm:superlinear}. We omit a direct proof of Corollary~\ref{cor:pers-det-trees} since it is an immediate consequence of Theorem~\ref{thm:main}. Aside from results presented in Section~\ref{sec:auxiliary-lemmas}, our proofs of Theorem~\ref{thm:persistence} and Theorem~\ref{thm:unique} are self contained. In particular, these results do not require prior knowledge about Crump-Mode-Jagers processes. 
\end{enumerate}

\section{Description of the models and statements of results} \label{sec:statements}
Suppose that $(\mathcal{T}_{n})_{n \in \mathbb{N}_0}$ denotes a sequence of directed trees on a vertex set $V$. In this paper, we are generally interested in whether the nodes of maximal out-degree in $(\mathcal{T}_{n})_{n \in \mathbb{N}_0}$ appear `early' or `late'. In this regard, for a directed tree $T$ we denote by $\outdeg{u, T}$ the out-degree of a node $u$ in $T$. It will also be helpful to use the convention that, if $u \notin T$, we have $\outdeg{u, T} = -\infty$. We then have the following definition:
\begin{defn} \label{def:persistence}
Suppose that $(\mathcal{T}_{n})_{n \in \mathbb{N}_{0}}$ is a sequence of directed trees. We say $u \in \bigcup_{n \in \mathbb{N}_{0}} \mathcal{T}_{n}$ is a \emph{persistent hub} if $\outdeg{u, \mathcal{T}_{n}} = \max_{v \in \mathcal{T}_{n}} \outdeg{v, \mathcal{T}_{n}}$ for all but finitely many $n \in \mathbb{N}_0$.  A persistent hub $u$ is unique if it is a persistent hub, and for any other persistent hub $u'$, we have $u = u'$. If $u$ is a persistent hub, or unique persistent hub we say $(\mathcal{T}_{n})_{n \in \mathbb{N}_{0}}$ contains a persistent hub, or unique persistent hub respectively.
\end{defn}

\subsection{Description of CMJ processes and related results} \label{sec:cmj-results}

A Crump-Mode-Jagers process $(\mathscr{T}_{t})_{t \geq 0}$ represents a total population of individuals initiated by a single ancestor, where, for any $t > 0$, $\mathscr{T}_{t}$ represents the population born before time $t$. We consider \emph{individuals} as being labelled according to their lineage, encoded by elements of the infinite \emph{Ulam-Harris} tree $\mathcal{U} : = \bigcup_{n \geq 0} \mathbb{N}^{n}$. The set $\mathbb{N}^{0} := \{\varnothing\}$ represents the ancestral \emph{root} individual $\varnothing$. We denote elements $u \in \mathcal{U}$ as a tuple, so that, if $u = (u_{1}, \ldots, u_{k}) \in \mathbb{N}^{k}, k \geq 1$, we write $u = u_{1} \cdots u_{k}$. An individual $u = u_1u_2\cdots u_k$ is to be interpreted recursively as the $u_k$th \emph{child} of the individual $u_1 \cdots u_{k-1}$. For example, the elements of $\mathbb{N}$, $1, 2, \ldots$ represent the offspring of $\varnothing$. We label elements of $\mathcal{U}$ with values in $[0, \infty]$, representing \emph{birth-times}. In particular, associated with each $u \in \mathcal{U}$ is a collection of random variables $(X(uj))_{j \in \mathbb{N}} \in [0,\infty]^\mathbb{N}$. We think of $X(uj)$ as the \emph{displacement} or \emph{waiting time} between the $(j-1)$th and $j$th child of $u$. We then define 
the random function $\mathcal{B}\colon \mathcal{U} \rightarrow [0, \infty]$ recursively as follows:  
\[\mathcal{B}(\varnothing) : = 0 \quad \text{and for $u \in \mathcal{U}, i \in \mathbb{N}$,} \quad \mathcal{B}(ui) := \mathcal{B}(u) + \sum_{j=1}^{i} X(uj).\]
For each $u \in \mathcal{U}$ we think of the value $\mathcal{B}(u)$ as its `birth time'.

An assumption that we apply throughout, that is implicit in the study of CMJ branching processes, is that the random variables $(X(uj))_{j \in \mathbb{N}}$ are i.i.d for different $u \in \mathcal{U}$. We use the notation $(X_{j})_{j\in \mathbb{N}}, (X'_{j})_{j\in \mathbb{N}}$ to denote generic independent sequences of random variables with 
\begin{equation} \label{eq:def-ex}
((X(uj))_{j \in \mathbb{N}}) \sim (X_{j})_{j \in \mathbb{N}} \sim (X'_{j})_{j \in \mathbb{N}}, \quad \text{for all } u \in \mathcal{U}.
\end{equation}

We use $|\cdot|$ to measure the \emph{length} of a tuple $u$, so that, if $u = \varnothing$ we set $|u| = 0$, whilst if $u = u_{1} \cdots u_{k}$ then $|u| = k$.
Given $\ell \leq |u|$, we set $u_{|_\ell} := u_{1} \cdots u_{\ell}$, and say $u_{\ell}$ is an \emph{ancestor} of $u$. It will be helpful to equip $\mathcal{U}$ with the lexicographic total order $\leq_{L}$: given elements $u, v$ we say $u \leq_{L} v$ if either $u$ is a ancestor of $v$, or $u_{\ell} < v_{\ell}$ where $\ell = \min \left\{i \in \mathbb{N}: u_{i} \neq v_{i} \right\}$. 
 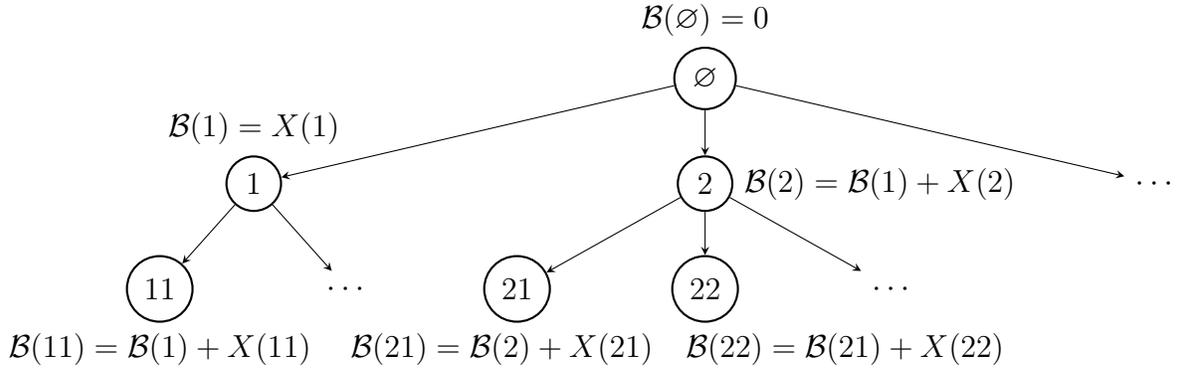
\begin{figure}[H]
\centering
\captionsetup{width=0.85\textwidth}
\begin{tikzpicture}[level distance=14mm,
   level 1/.style={sibling distance=60mm},
   level 2/.style={sibling distance=25mm},
   level 3/.style={sibling distance=10mm} 
   ]
   \node[circle, draw, thick,  label={$\mathcal{B}(\varnothing) = 0$}] (0) {$\varnothing$}
    child {node[circle, draw, thick, label=above:{$\mathcal{B}(1) = X(1)$}] {$1$} edge from parent[-stealth]
      child {node[circle, draw, thick, label=below:{$\mathcal{B}(11) = \mathcal{B}(1) + X(11)$}] {$11$}
      } 
      child {node {$\cdots$}}
    }
    child {node[circle, draw, thick, label=right:{$\mathcal{B}(2) = \mathcal{B}(1) + X(2)$}] {$2$} edge from parent[-stealth]
      child {node[circle, draw, thick, label=below:{$\mathcal{B}(21) = \mathcal{B}(2) + X(21) \phantom{ve}$}] {$21$}}
      child {node[circle, draw, thick, label=below:{$\phantom{mathcalffdrgttgtrf} \mathcal{B}(22) = \mathcal{B}(21) + X(22) $}] {$22$}}
      child {node {$\cdots$}}
    }
    child {node {$\cdots$} edge from parent[-stealth]}
    ;
\end{tikzpicture}
\caption{A illustration of the way birth times are assigned to individuals in the first three generations of the process. Note that birth times are increasing on paths directed away from the root. }
\end{figure}
 For each $t \in [0, \infty]$, we set $\mathscr{T}_{t} = \{x \in \mathcal{U} \colon \mathcal{B}(x) \leq t\}$ and denote by $(\mathcal{F}_{t})_{t \geq 0}$ the natural filtrations generated by $(\mathscr{T}_{t})_{t \geq 0}$. If a subset $T \subset \mathcal{U}$ is such that, for any $u \in T$, we also have $u_{|_{\ell}} \in T$, for each $\ell \leq |u|$, note that one may view $T$ as a \emph{directed tree} in the natural way, connecting nodes with edges directed outwards to their children. Therefore, the process $(\mathscr{T}_{t})_{t \geq 0}$ yields an increasing family of directed trees, where, if $s < t$, we have $\mathscr{T}_{s} \subseteq \mathscr{T}_{t}$. In relation to the process $(\mathscr{T}_{t})_{t \geq 0}$, we define the stopping times $(\tau_{k})_{k \in \mathbb{N}_{0}}$ such that 
 \begin{equation} \label{eq:tau-k-def} 
 \tau_{k} := \inf\{t \geq 0\colon |\mathscr{T}_{t}| \geq k\}.
 \end{equation}
Note that the values $(\tau_{k})_{k \in \mathbb{N}_0}$ describe the times in which changes, or `jumps' occur in the process $(\mathscr{T}_{t})_{t \geq 0}$.
The process $(\mathscr{T}_{\tau_{n}})_{n \in \mathbb{N}_{0}}$ therefore describes the total family of discrete trees appearing as jumps in the process $(\mathscr{T}_{t})_{t \geq 0}$ before $\tau_{\infty}:= \lim_{k \to \infty} \tau_{k}$. 
\begin{figure}[H]
\centering
\captionsetup{width=0.85\textwidth}
\begin{tikzpicture}[level distance=14mm,
   level 1/.style={sibling distance=60mm},
   level 2/.style={sibling distance=25mm},
   level 3/.style={sibling distance=15mm}]
   \node[circle, draw, thick,  label={$0$}] (0) {$\varnothing$}
    child {node[circle, draw, thick, label=left:{$0.5$}] {$1$} edge from parent[-stealth]
      child {node[circle, draw, thick, label=left:{$0.6$}] {$11$}
      }
      child {node[circle, draw, thick, label=left:{$1.1$}] {$12$}
        child {node[circle, draw, thick, label=below:{$1.8$}] {$121$}}
    }}
    child {node[circle, draw, thick, label=left:{$1.2$}] {$2$} edge from parent[-stealth]
      child {node[circle, draw, thick, label=left:{$1.7$}] {$21$}}
      child {node[circle, draw, thick, label=left:{$1.74$}] {$22$}
        child {node[circle, draw, thick, label=below:{$1.9$}] {$221$}}
        child {node[circle, draw, thick, label=below:{$1.99$}] {$222$}}
      }
    }
    ;
\end{tikzpicture}
\caption{A possible sample of the process $(\mathscr{T}_{t})_{t\geq 0}$ at time $t = 1.99$, with birth times labelled. In this case $\tau_{9} = 1.99$. }
\end{figure}
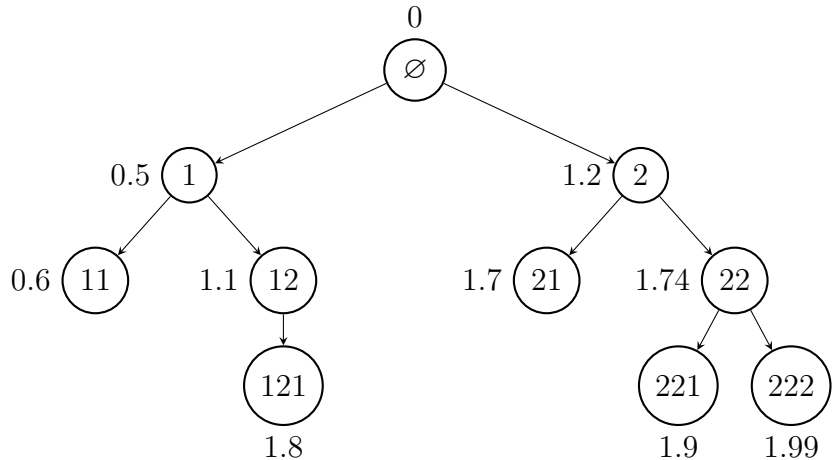
In this paper we assume the following throughout. Recall the definitions of $(X_{j})_{j \in \mathbb{N}}, (X'_{j})_{j \in \mathbb{N}}$ from~\eqref{eq:def-ex}.

\begin{assumption} \label{ass:gen-ass}
We assume throughout this paper that, that the collection $(X_{i})_{i \in \mathbb{N}}$ defined in~\eqref{eq:def-ex} satisfies the following:
\begin{enumerate}
    \item \label{item:mutual-independence} The values $(X_i)_{i\in\N}$ are mutually independent of each other. 
    \item \label{item:finiteness} For each $j \in \mathbb{N}$ we have $X_{j} < \infty$, almost surely.
    \item \label{item:non-zero-explosion} We have 
    $\sum_{j=1}^{\infty} \prod_{i=1}^{j} \Prob{X_{i} = 0} < 1$. 
\end{enumerate}
\end{assumption}
\begin{rmq}
Some comments about Assumption~\ref{ass:gen-ass} are the following:
\begin{enumerate}
    \item Item~\ref{item:mutual-independence} of Assumption~\ref{ass:gen-ass} asserts that the waiting times between births of children of an individual $u$ are independent of each other. This may not always be desirable, for example, one might expect these waiting times to be correlated, depending on a `random attribute' associated with $u$, as is the case in inhomogeneous models such as~\cite{dereich-mailler-morters, inhom-sup-pref-attach}. 
    \item Item~\ref{item:finiteness} of Assumption~\ref{ass:gen-ass} implies that every individual produces infinitely many total offspring as time tends to infinity, hence that the process is supercritical. 
    \item Note that 
\[
\E{\left|\left\{i \in \mathbb{N}\colon \mathcal{B}(i) = 0 \right\} \right|} = \sum_{j=1}^{\infty} \Prob{\left|\left\{i \in \mathbb{N}\colon \mathcal{B}(i) = 0 \right\} \right| \geq j} = \sum_{j=1}^{\infty} \prod_{i=1}^{j} \Prob{X_{i} = 0}. 
\]
Therefore, Item~\ref{item:non-zero-explosion} of Assumption~\ref{item:non-zero-explosion} shows that the tree of individuals born `instantaneously' is a tree associated with a sub-critical branching process, hence is finite almost surely. Thus, if $\tau_{k}$ are as defined in~\eqref{eq:tau-k-def}, this assumption removes the degenerate case that $\tau_{k} = 0$ for all $k \in \mathbb{N}$.
\end{enumerate}
{\small\ensymboldremark }
\end{rmq}

In the following theorem, recall that, given a sequence $(S_{j})_{j \in \mathbb{N}}$ of mutually independent random variables, the series $\sum_{j=1}^{\infty}S_{j}$ converges with probability zero or one, and criteria for this convergence are given by the well-known Kolmogorov three series theorem (c.f. Lemma~\ref{lem:kolmogorov}). 

\begin{thm} \label{thm:persistence}
    Assume $(X_{j})_{j \in \mathbb{N}}, (X'_{j})_{j \in \mathbb{N}}$ are defined as in~\eqref{eq:def-ex} and Assumption~\ref{ass:gen-ass} is satisfied. Then, in the process $(\mathscr{T}_{\tau_{n}})_{n \in \mathbb{N}_0}$, we have the following: 
    \begin{enumerate}
        \item \label{item:non-persistence}  If the series $\sum_{i=1}^{\infty} (X_{i} - X'_{i})$ diverges almost surely, then, almost surely, $(\mathscr{T}_{\tau_{n}})_{n \in \mathbb{N}_0}$ does not contain a persistent hub.
        \item \label{item:persistent-hub} Suppose that, for $\alpha > 0$ and $K \in \mathbb{N}$ we have
        \begin{equation} \label{eq:persistent}
             \sum_{j=1}^{\infty} \E{e^{-\alpha \sum_{i=1}^{j} X_{i}}} < 1 \quad \text{and} \quad \prod_{i=K+1}^{\infty}\E{e^{\alpha (X'_{i} - X_{i}) }} < \infty.
        \end{equation}
        Then, almost surely, $(\mathscr{T}_{\tau_{n}})_{n \in \mathbb{N}_0}$ contains a persistent hub.  
    \end{enumerate} 
\end{thm}

\begin{rmq}
    The proof of Theorem~\ref{thm:persistence} uses a novel proof technique that does not rely on classical theory regarding CMJ processes. However, parts of the proof draw inspiration from the proof of~\cite[Theorem~2.5]{inhom-sup-pref-attach}, and the proof of~\cite[Theorem~1.1]{oliveira-spencer}. See also Remark~\ref{rem:bas} below. {\small\ensymboldremark }
\end{rmq}

\begin{rmq} \label{rem:imp-0}
Foundational results related to CMJ branching processes (see, e.g., \cite{nerman_81, nerman-jagers-84, jagers-nerman89, jagers-nerman-96, olofsson-x-log-x, iksanov2023asymptotic}) often assume the existence of $\alpha' >0$ (often called a \emph{Malthusian parameter}) such that 
\begin{equation} \label{eq:malth-ass}
\sum_{j=1}^{\infty} \E{e^{-\alpha' \sum_{i=1}^{j} X_{i}}} = 1.
\end{equation} 
The first condition in~\eqref{eq:persistent} is (at least morally) a weaker condition, since, in principle, it is possible that the map $\lambda \mapsto \sum_{j=1}^{\infty} \E{e^{-\lambda \sum_{i=1}^{j} X_{i}}}$ is discontinuous. In particular, it is unclear whether or not there exist counter-examples where
\[
\lambda' := \inf\left\{\lambda > 0\colon \sum_{j=1}^{\infty} \E{e^{-\lambda \sum_{i=1}^{j} X_{i}}} < \infty \right\} \quad \text{but} \quad \sum_{j=1}^{\infty} \E{e^{-\lambda' \sum_{i=1}^{j} X_{i}}} < 1,
\]
in which case~\eqref{eq:malth-ass} cannot be satisfied. Such counter-examples, however, are known to exist when the values $(X_{i})_{i \in \mathbb{N}}$ are correlated, depending on a random `weight' (see, e.g.,~\cite{dereich-mailler-morters, rec-trees-fit}). {\small\ensymboldremark }
\end{rmq}

\begin{rmq} \label{rem:general-conj}
   We believe~\eqref{eq:persistent} is not optimal. A more general conjecture, inspired by~\cite[Theorem~1.4]{leadership} is that, under Assumption~\ref{ass:gen-ass}, a persistent hub emerges with probability zero or one, and with probability one if and only if $\sum_{i=1}^{\infty}(X_{i} - X'_{i})$ converges almost surely. Note that Assumption~\ref{ass:gen-ass} and the condition from~\eqref{eq:persistent} that 
   \[\prod_{i=K+1}^{\infty}\E{e^{\alpha (X'_{i} - X_{i}) }} < \infty \quad \text{imply that} \quad \sum_{i=1}^{\infty}(X_{i} - X'_{i}) \quad \text{converges almost surely.} \]
   This fact is proven explicitly, and used, in the proof of Item~\ref{item:persistent-hub} of Theorem~\ref{thm:persistence}. {\small\ensymboldremark }
\end{rmq}

\begin{rmq} \label{rem:bas}
    Suppose that Assumption~\ref{ass:gen-ass} is satisfied, and that for each $j \in \mathbb{N}$, we have $\mu_{j} := \E{\sum_{i=j+1}^{\infty} X_{i}} < \infty$. Moreover, assume that for some $c > 0$
    \begin{equation} \label{eq:bas-assumptions}
    \sum_{j=1}^{\infty} \E{e^{-c \mu_{j}^{-1} \sum_{i=1}^{j} X_{i}}} < \infty \quad \text{and} \quad \limsup_{j \to \infty} \E{e^{c \mu_{j}^{-1}\sum_{i=j+1}^{\infty} X_{i} }} < \infty.
    \end{equation}
    Then, the result~\cite[Theorem~2.5]{inhom-sup-pref-attach} implies that the tree $\bigcup_{n=1}^{\infty} \mathscr{T}_{\tau_{n}}$ contains a node of infinite degree\footnote{Note that Items~1, 3 and 4 of \cite[Assumption~2.2]{inhom-sup-pref-attach} are implied by Assumption~\ref{ass:gen-ass}, and Equation~\eqref{eq:bas-assumptions} implies that Items 2 and 5 of \cite[Assumption~2.2]{inhom-sup-pref-attach} are satisfied with $Y_{n} = \sum_{i=n+1}^{\infty} X_{i}$.}.   
     We claim, without proof, that a similar approach to the proof of~\cite[Theorem~2.5]{inhom-sup-pref-attach} actually shows that $(\mathscr{T}_{\tau_{n}})_{n \in \mathbb{N}}$ is persistent under the same assumptions. As a comment for the reader familiar with some of the technical details of~\cite{inhom-sup-pref-attach}, this works by replacing the event ``$a$ explodes before each of its ancestors'', by the event ``$a$ catch up to each of its ancestors before an ancestor explodes''. {\small\ensymboldremark }
\end{rmq}

The following theorem provides criteria for persistent hubs to be unique. 
\begin{thm} \label{thm:unique}
Assume $(X_{j})_{j \in \mathbb{N}}, (X'_{j})_{j \in \mathbb{N}}$ are defined as in~\eqref{eq:def-ex} and Assumption~\ref{ass:gen-ass} is satisfied. Moreover, assume~\eqref{eq:persistent} is satisfied 
and, in addition, one of the following conditions is satisfied:
\begin{enumerate}
    \item \label{item:borel-cantelli-ass} We have
        \begin{equation} \label{eq:borel-cantelli}
        \sum_{k=1}^{\infty} \Prob{0 \leq \sum_{j=1}^{k} (X'_{j} - X_{j}) - \sum_{j=1}^{k} X_{j} < X_{k+1}} < \infty.
        \end{equation}
    \item \label{item:conv-series-and-borel} It is the case that both of the following conditions are satisfied:
    \begin{enumerate}
        \item We have $\sum_{j=1}^{\infty} \Prob{X_{j} - X'_{j} \neq 0} = \infty$, or for some $j \in \mathbb{N}$, we have $\Prob{X_{j} \neq X'_{j}} = 1$.
        \item For any $\eps > 0$ we have $\sum_{j=1}^{\infty} \Prob{X_{j} > \eps} < \infty$. 
    \end{enumerate}
\end{enumerate}    
Then, almost surely, in the process $(\mathscr{T}_{\tau_{n}})_{n \in \mathbb{N}_{0}}$,
there is a unique persistent hub. 
\end{thm}
\begin{rmq}
    Theorem~\ref{thm:unique} is closely related~\cite[Theorem~1.4 \& Corollary~1.8]{leadership}, but we include a proof for completeness. {\small\ensymboldremark }
\end{rmq}

\subsection{Description of preferential attachment trees and related results} \label{sec:preferential}
One particular scenario in which Theorem~\ref{thm:persistence} and Theorem~\ref{thm:unique} may be applied is to variants of preferential attachment trees where new nodes connect to existing nodes with probability proportional to a random function of the degree of that node. In particular, suppose that we construct a sequence of directed trees in the following way. For each $j \in \mathbb{N}_{0}$ suppose that $(F_{j}(k))_{k \in \mathbb{N}_{0}}$ is a collection of mutually independent random variables, identically distributed across $j \in \mathbb{N}_{0}$.  
At time $0$ we define the tree $\mathcal{T}_{0}$ consisting of a single node labelled $0$. Then, recursively, given the tree $\mathcal{T}_{k}$ at time $k \in \mathbb{N}_0$, and the values of the random variables $(F_{i}(\outdeg{i, \mathcal{T}_{k}}))_{i \leq k}$, 
\begin{enumerate}
    \item Form $\mathcal{T}_{k+1}$ by sampling a node $j$ from $\mathcal{T}_k$ with probability 
    \begin{equation} \label{eq:discrete-dynamics}
        \frac{F_{j}(\outdeg{j, \mathcal{T}_{k}})}{\mathcal{Z}_{k}}, \quad \text{ with } \quad \mathcal{Z}_{k} := \sum_{j=0}^{k} F_{j}(\outdeg{j, \mathcal{T}_{k}})
    \end{equation}
and connecting $j$ with an edge directed outwards to a new node labelled $k+1$.
    \item Sample the random variables $F_{j}(\outdeg{j, \mathcal{T}_{k+1}}) = F_{j}(\outdeg{j, \mathcal{T}_{k}} +1)$ and $F_{k+1}(\outdeg{k+1, \mathcal{T}_{k+1}})$. As the degrees of all the other nodes remain unchanged, this allows one to determine the values $(F_{i}(\outdeg{i, \mathcal{T}_{k+1}}))_{i \leq k+1}$ for the next step in the process. 
\end{enumerate}
Let $(F(k))_{k \in \mathbb{N}_0}$ denotes an independent sequence of random variables, with $(F(k))_{k \in \mathbb{N}_0} \sim (F_{0}(k))_{k \in \mathbb{N}_0}$.

\begin{thm} \label{thm:main}
    Suppose that $(\mathcal{T}_{i})_{i \in \mathbb{N}_{0}}$ is a generalised preferential attachment tree. Then,
    \begin{enumerate} 
        \item \label{item:non-preference} If \[\sum_{j=0}^{\infty} \frac{1}{F(j)^2} = \infty \quad \text{almost surely},\] then $(\mathcal{T}_{i})_{i \in \mathbb{N}_{0}}$ almost surely does not have a persistent hub. 
        \item \label{item:preference-one} Suppose that the following two conditions are satisfied: 
        \begin{enumerate}
            \item There exists a sequence $(x_{j})_{j \in \mathbb{N}_0}$ with $\sum_{j=0}^{\infty} \frac{1}{x_{j}^2} < \infty$ such that,
				\begin{equation} \label{eq:det-lower-bound}           
             		F(j) \geq x_{j} \quad \text{almost surely.}
             	\end{equation} 
            \item There exists a $\lambda > 0$ such that 
                \begin{equation} \label{eq:lambda-eq}
                \sum_{i=0}^{\infty} \prod_{j=0}^{i}\E{\frac{F(j)}{F(j) + \lambda}} < \infty.
                \end{equation}
        \end{enumerate}
        Then $(\mathcal{T}_{i})_{i \in \mathbb{N}_{0}}$ has a unique persistent hub almost surely. 
        \end{enumerate}
\end{thm}

\begin{rmq}
    Similar to Remark~\ref{rem:general-conj}, we believe Theorem~\ref{thm:main} can be improved. 
    Inspired by~\cite[Theorem~1.10]{leadership} we conjecture that, a unique persistent hub emerges with probability zero or one, and with probability one if and only if $\sum_{j=0}^{\infty} \frac{1}{F(j)^2} < \infty$ almost surely. {\small\ensymboldremark }
\end{rmq}

Suppose that the values $(F(j))_{j \in \mathbb{N}_{0}}$  are given by a deterministic sequence $(f(j))_{j \in \mathbb{N}_{0}}$. Then, by choosing the sequence $(x_{j})_{j\in \mathbb{N}_{0}} = (f(j))_{j \in \mathbb{N}_{0}}$ in Theorem~\ref{thm:main}, we have the following immediate corollary, whose proof we omit. 

\begin{cor} \label{cor:pers-det-trees}
Suppose that $(\mathcal{T}_{i})_{i \in \mathbb{N}_{0}}$ is a generalised preferential attachment tree, with deterministic values $(f(j))_{j \in \mathbb{N}_{0}}$. Then
\begin{enumerate}
\item If $\sum_{j=0}^{\infty} \frac{1}{f(j)^2} = \infty$, then $(\mathcal{T}_{i})_{i \in \mathbb{N}_{0}}$ almost surely does not have a persistent hub.
\item \label{item:det-persistent} If $\sum_{j=0}^{\infty} \frac{1}{f(j)^2} < \infty$ and for some $\lambda > 0$, 
\begin{equation} \label{eq:gen-malth}
\sum_{i=0}^{\infty} \prod_{j=0}^{i}\frac{f(j)}{f(j) + \lambda} < \infty,
\end{equation} 
then $(\mathcal{T}_{i})_{i \in \mathbb{N}_{0}}$ has a unique persistent hub almost surely. 
\end{enumerate}
\end{cor}
{\small \hfill $\blacksquare$}
\begin{rmq} \label{rem:imp-1}
    Corollary~\ref{cor:pers-det-trees} improves on an existing result related to the preferential attachment tree by Banerjee and Bhamidi in~\cite{banerjee-bhamidi}. One assumption from that paper is that $\inf_{i \geq 0} f(i) > 0$ - this is not needed for Corollary~\ref{cor:pers-det-trees}.  Another assumption from~\cite{banerjee-bhamidi} is that there exists some $\lambda' > 0$ such that
    \begin{equation} \label{eq:malth-ban-bham}
    1 < \sum_{i=0}^{\infty} \prod_{j=0}^{i}\frac{f(j)}{f(j) + \lambda'} < \infty. 
    \end{equation}
    This assumption ensures the existence of a solution to~\eqref{eq:malth-ass}, so that classical results related to CMJ branching processes may be applied. It should be noted, however, that in~\cite{banerjee-bhamidi}, this classical theory is used to derive asymptotics of the maximal degree when $\sum_{j=0}^{\infty} \frac{1}{f(j)^2} = \infty$ and~\eqref{eq:malth-ban-bham} is satisfied, whilst in~\cite{banerjee-bhamidi-centrality} criteria are given for persistence of more general centrality measures than degree centrality (persistent hubs).
    
    It is not immediately clear whether or not there exist examples where~\eqref{eq:gen-malth} is satisfied, but~\eqref{eq:malth-ban-bham} is not. However,~\eqref{eq:gen-malth} is easier to verify in many cases, as is shown by Example~\ref{ex:bounded-linear} below. {\small\ensymboldremark }
\end{rmq}

\begin{ex} \label{ex:bounded-linear}
Suppose that, for some constant $C_{0} > 0$, for all $i \in \mathbb{N}_{0}$ we have 
\begin{equation} \label{eq:sub-linear-growth}
f(i) \leq C_{0} (i+1).
\end{equation}
Then,~\eqref{eq:gen-malth} is always satisfied. Indeed, for any $\lambda > 0$, we have 
\begin{equation} \label{eq:gamma-bound}
\sum_{i=0}^{\infty} \prod_{j=0}^{i} \frac{f(j)}{f(j) + \lambda} \leq \sum_{i=0}^{\infty} \prod_{j=0}^{i} \frac{C_0 (j+1)}{C_{0}(j+1) + \lambda} = \sum_{i=0}^{\infty} \frac{\Gamma(i+2) \Gamma(\lambda/C_0)}{\Gamma(i+2+\lambda/C_0)}. 
\end{equation}
It is well-known, for example by Stirling's approximation, that we can find a constant $C_1 > 0$ such that, for all $i \in \mathbb{N}_{0}$, we have $\frac{\Gamma(i+2)}{\Gamma(i+2+\lambda/C_0)} \leq C_{1} (i+1)^{-\lambda/C_{0}}$. Therefore, for any $\lambda > 2C_0$, say, we may bound the right-side of~\eqref{eq:gamma-bound} so that 
\[
\sum_{i=0}^{\infty} \prod_{j=0}^{i} \frac{f(j)}{f(j) + \lambda} \leq C_1 \Gamma(\lambda/C_0) \sum_{i=1}^{\infty} (i+1)^{-\lambda/C_{0}} < C_1 \Gamma(\lambda/C_0) \sum_{i=1}^{\infty} (i+1)^{-2} < \infty,
\]
thus confirming~\eqref{eq:gen-malth}. In order to show~\eqref{eq:malth-ban-bham} we would have to bound the left side of~\eqref{eq:gamma-bound} from below, which is more difficult in general. {\small\ensymboldremark }
\end{ex}
If~\eqref{eq:sub-linear-growth} is not satisfied, it must be the case that
\begin{equation} \label{eq:limsup-infinite}
\limsup_{i \to \infty} \frac{f(i)}{i+1} = \infty.
\end{equation}
This indicates that there is somehow a `high' degree of reinforcement, hence one would expect it to be more likely that there is a persistent hub. A partial result in this direction is Theorem~\ref{thm:superlinear} below. 
\begin{rmq}
 Some previous results in the literature concerning cases where~\eqref{eq:limsup-infinite} may be satisfied are the following:
\begin{enumerate}
    \item In~\cite{galashin}, Galashin showed that whenever $f(x)$ is convex (as a function from $\mathbb{R} \rightarrow \mathbb{R}$) and unbounded, there exists a persistent hub. 
    \item In~\cite{oliveira-spencer} the authors showed that a more extreme effect emerges when $f(n) = (n+1)^{p}$ for $p > 1$: the infinite tree $\bigcup_{i=1}^{\infty} \mathcal{T}_{i}$ contains a unique node of infinite degree, and implicitly showed that the degree of every other node in $\bigcup_{i=1}^{\infty} \mathcal{T}_{i}$ is bounded almost surely, hence implying that $(\mathcal{T}_{i})_{i \in \mathbb{N}}$ is persistent.
    \item As a particular application of~\cite[Theorem~3.4]{inhom-sup-pref-attach} (see also \cite[Remark~3.19]{inhom-sup-pref-attach}), the result of~\cite{oliveira-spencer} can be extended to show the following. Suppose $\mu_{n} := \sum_{j=n}^{\infty} \frac{1}{f(j)}$. Then, assuming $\mu_{0} < \infty$ and for some $c < 1$,
    \begin{equation} \label{eq:laplace-sum}
    \sum_{i=0}^{\infty} \prod_{j=0}^{\infty} \frac{f(j)}{f(j) + c\mu_{i}^{-1}} < \infty,
    \end{equation}
    $\bigcup_{i=1}^{\infty} \mathcal{T}_{i}$ contains a unique node of infinite degree. We claim, without proof, that~\eqref{eq:laplace-sum} is satisfied under the assumption that for all $n \in \mathbb{N}_0$, we have $f(n) \geq C (n + 1) \log{(n+2)}^{\alpha}$ for $\alpha > 2$ and $C > 0$. Moreover, in a similar manner to Remark~\ref{rem:bas}, we also claim that these results may be extended to show that, under the same conditions, $(\mathcal{T}_{i})_{i \in \mathbb{N}_0}$ contains a unique persistent hub. We omit proofs of both these claims, since they are anyway superseded by Theorem~\ref{thm:superlinear}. {\small\ensymboldremark }
\end{enumerate} 
\end{rmq}

\begin{thm} \label{thm:superlinear}
Suppose that $(\mathcal{T}_{i})_{i \in \mathbb{N}_{0}}$ is a generalised preferential attachment tree, with deterministic values $(f(j))_{j \in \mathbb{N}_{0}}$. Moreover, assume that, for some $\kappa > 0$, we have 
\begin{equation} \label{eq:controlled-superlinear}
\max_{i \leq n} \frac{f(i)}{i+1} \leq \kappa \frac{f(n)}{n+1}.
\end{equation}
Then, almost surely, $(\mathcal{T}_{i})_{i \in \mathbb{N}_{0}}$ contains a unique persistent hub. 
\end{thm}

\begin{rmq}
Unlike the other results in this section, Theorem~\ref{thm:superlinear} does not have an analogue for more general trees associated with CMJ branching processes, rather relies more heavily on the dynamics of the generalised preferential attachment tree. Part of the proof of Theorem~\ref{thm:superlinear}, notably Claim~\ref{clm:super-martingale}, uses a similar martingale argument to~\cite[Proposition~12]{galashin} to show that the maximal degree in the process $(\mathcal{T}_{i})_{i \in \mathbb{N}}$ grows sufficiently quickly. {\small\ensymboldremark }
\end{rmq}

\begin{rmq}
    Note that~\eqref{eq:controlled-superlinear} implies $f(n) \geq \frac{f(0)}{\kappa} (n+1)$ and thus $\sum_{n=0}^{\infty}1/f(n)^2 < \infty$. Therefore it is not necessary to impose this summability as an extra condition, unlike Item~\ref{item:det-persistent} of Corollary~\ref{cor:pers-det-trees}. {\small\ensymboldremark }
\end{rmq}

\begin{exs} \label{exs:imp-2}
Some example cases where~\eqref{eq:controlled-superlinear} is satisfied are the following:
\begin{enumerate}
    \item Any function $f$ such that $f(n)/(n+1)$ is non-decreasing in $n \in \mathbb{N}$. This includes many `barely faster-than-linear' examples of $f$, such as $f(n) = (n+1) \log(\log(n+3))$, for $n \in \mathbb{N}_{0}$. 
    \item Any function $f$ such that $f(x)$ is convex and unbounded (when extended to a function $\mathbb{R} \rightarrow \mathbb{R}$). Indeed, convexity implies that the function $f(x) - f(0)$ is super-additive on the positive reals. Thus, by unboundedness, and Fekete's lemma, 
    \begin{equation} \label{eq:fekete}
     \lim_{n \to \infty} \frac{f(n)}{n+1} = \sup_{\ell \in \mathbb{N}_{0}} \frac{f(\ell)}{\ell +1} > 0,   
    \end{equation}
    which implies~\eqref{eq:controlled-superlinear}.
    \item One may construct many examples where~\eqref{eq:controlled-superlinear} where $f(n)$ is not monotone, let alone convex. Hence~\eqref{eq:controlled-superlinear} is a strictly weaker condition than the condition of convexity and unboundedness required in~\cite{galashin}. 
    For example, if 
    \[
    f(n) = \begin{cases}
        (n+1)^2 & \text{if $n = 1$ or $n$ is even} \\
        n^{2} - 1 & \text{otherwise,}
    \end{cases} \]
    its easy to check that one can set $\kappa = 2$ in~\eqref{eq:controlled-superlinear}. {\small\ensymboldremark }
\end{enumerate}
\end{exs}

\section{Proofs of results} \label{sec:proofs}
\subsection{Some auxiliary lemmata}
\label{sec:auxiliary-lemmas}
In this section we collect some lemmas that will be useful in the sequel. The proofs of the first three are omitted, but we include the proof of the fourth lemma. 
\begin{lem}[Kolmogorov three series theorem, e.g.{~\cite[Theorem~2.5.8, page~73]{durrett}}] \label{lem:kolmogorov}
 For a sequence of mutually independent random variables $(S_{j})_{j \in \mathbb{N}}$, let $C > 0$ be given. Then the series $\sum_{j=1}^{\infty} S_{j}$ converges almost surely if and only if
 \begin{equation} \label{eq:kolmogorov-three-series}
     \sum_{j=1}^{\infty} \Prob{\left|S_{j}\right| > C} < \infty, \quad \sum_{j=1}^{\infty} \E{S_{j} \mathbf{1}_{\left|S_{j}\right| \leq C}} < \infty \quad \text{and} \quad \sum_{j=1}^{\infty} \var{S_{j} \mathbf{1}_{\left|S_{j}\right| \leq C}} < \infty. 
 \end{equation}
\hfill $\blacksquare$
\end{lem}

Next, we state without proof some results from~\cite{leadership} related to series of independent random variables.

\begin{lem}[{\cite[Theorem~1.13]{leadership}}] \label{lem:no-atom}
     Suppose that $(S_{j})_{j \in \mathbb{N}}$ is a sequence of mutually independent random variables such that $\sum_{j=1}^{\infty} S_{j}$ converges almost surely. Then the distribution of $\sum_{j=1}^{\infty} S_{j}$ contains an atom on $\mathbb{R}$ if and only if, for some collection $(c_{j})_{j \in \mathbb{N}} \in \mathbb{R}^{\mathbb{N}}$
        \begin{equation} \label{eq:atomic-sum}
            \forall j \in \mathbb{N} \quad \Prob{S_{j} = c_{j}} > 0 \quad \text{ and } \quad \sum_{j=1}^{\infty} \Prob{S_{j} \neq c_{j}} < \infty. 
        \end{equation}
\end{lem}
\hfill $\blacksquare$

\begin{lem}[{\cite[Proposition~1.15]{leadership}}] \label{lem:fluctuating-sum}
    Suppose~$(S_{j})_{j \in \mathbb{N}}$ is a sequence of mutually independent, symmetric random variables. If $\sum_{j=1}^{\infty} S_{j}$ diverges almost surely, then 
        \begin{equation} \label{eq:varying-sum}
        \limsup_{n \to \infty} \sum_{j=1}^{n} S_j = \infty \quad \text{and} \quad \liminf_{n \to \infty} \sum_{j=1}^{n} S_{j} = -\infty \quad \text{almost surely.}
        \end{equation}
\end{lem}
\hfill $\blacksquare$

Finally, we state an prove an inequality which bounds the probability that one random series of independent random variables `overshoots' another by a certain amount.  

\begin{lem} \label{lem:overtake-bound}
    Suppose that $(X_{i})_{i \in \mathbb{N}}$, $(X'_{i})_{i \in \mathbb{N}}$ and $Y$ are mutually independent random variables, with $(X_{i})_{i \in \mathbb{N}} \sim (X'_{i})_{i \in \mathbb{N}}$. Suppose that, for $k \in \mathbb{N}_{0}$, there exists $\lambda > 0$ such that
    \begin{equation} \label{eq:finite-mgf-mart}
    \prod_{i=k}^{\infty}\E{e^{\lambda  (X'_{i} - X_{i})}} < \infty.  
    \end{equation}
    Then, for any $\ell \in \mathbb{N}$
    \begin{equation} \label{eq:maximal-bound}
        \Prob{\exists j \in \mathbb{N}\colon Y + \sum_{i=1}^{k + j} X_{i} \leq \sum_{i=k}^{k+j} X'_{i}} \leq \left(\prod_{i=k}^{\infty}\E{e^{\lambda  (X'_{i} - X_{i})}}\right) \E{e^{- \lambda \left(Y + \sum_{i=1}^{k-1} X_{i}\right)}}.
    \end{equation}
\end{lem}

\begin{proof}[Proof of Lemma~\ref{lem:overtake-bound}]
Suppose that $[n] := \left\{1, \ldots, n\right\}$. Then, for $n \in \mathbb{N}$, and $\lambda$ satisfying Equation~\eqref{eq:finite-mgf-mart} note that
\begin{linenomath*}
\begin{align*}
&\Prob{\exists j \in [n]\colon Y + \sum_{i=1}^{k+j} X_{i} \leq \sum_{i=k}^{k+j} X'_{i}} = \Prob{\exists j \in [n]\colon \lambda  \left( \sum_{i=k}^{k+j} (X'_{i} - X_{i}) \right) - \lambda \left(Y + \sum_{i=1}^{k-1} X_{i}\right) \geq 0}
\\ &\hspace{3cm} = \Prob{\exists j \in [n]\colon \exp{\left(\lambda  \left( \sum_{i=k}^{k+j} (X'_{i} - X_{i}) \right) - \lambda \left(Y + \sum_{i=1}^{k-1} X_{i}\right) \right)} \geq 1}.
\end{align*}
\end{linenomath*}
Now, define $(M_{j})_{j \in \mathbb{N}_{0}}$ by $M_{0} := \exp{\left(- \lambda \left(Y + \sum_{i=1}^{k-1} X_{i} \right)\right)}$, and \[M_{i+1} = M_{i} \exp{\left(\lambda (X'_{k+i} - X_{k+i}) \right)}.\] Then, $M_{n} = \exp{\left(\lambda  \left( \sum_{i=k}^{k+n-1} (X'_{i} - X_{i}) \right) - \lambda \left(Y + \sum_{i=1}^{k-1} X_{i} \right)\right)}$ and
\begin{linenomath}
\begin{align}
\E{M_{i+1} \, | \, M_{0}, \ldots, M_{i}} & = M_{i} \E{\exp{\left(\lambda (X'_{k+i+1} - X_{k+i+1}) \right)}} \\ & \geq M_{i} \exp{\left(\lambda \E{(X'_{k+i+1} - X_{k+i+1})} \right)} =  M_{i}, 
\end{align}
\end{linenomath}
where we have applied Jensen's inequality. Combining this with Equation~\eqref{eq:finite-mgf-mart}, we deduce that the sequence $(M_{j})_{j \in \mathbb{N}_0}$ is a sub-martingale sequence. By Doob's sub-martingale inequality, 
\begin{equation} \label{eq:doob-usage}
    \Prob{\exists j \in [n]\colon \exp{\left(\lambda  \left( \sum_{i=k}^{k+j} (X'_{i} - X_{i}) \right) - \lambda \left(Y + \sum_{i=1}^{k-1} X_{i}\right)\right)} \geq 1} = \Prob{\max_{0 \leq j \leq n+1} M_{j} \geq 1} \leq \E{M_{n+1}}. 
\end{equation}
Taking monotone limits as $n \to \infty$, we deduce~\eqref{eq:maximal-bound}. 
\end{proof}

\subsection{Proofs of Theorem~\ref{thm:persistence} and Theorem~\ref{thm:unique}} \label{sec:cmj-proofs}

For the proof of Theorem~\ref{thm:persistence} and Theorem~\ref{thm:unique}, we first define some  terminology, and state and prove Lemmas~\ref{lem:non-expl} and~\ref{lem:finite-persistence-claims}. We then proceed with the proof of Theorem~\ref{thm:persistence}: the proof of Item~\ref{item:persistent-hub} in Section~\ref{sec:persistent-hub}, and the proof of Item~\ref{item:non-persistence} in Section~\ref{sec:non-persistence}. We prove Theorem~\ref{thm:unique} in Section~\ref{sec:unique}. 

First, we show that the first part of Equation~\eqref{eq:persistent} implies that the process $(\mathscr{T}_{t})_{t \geq 0}$ is almost surely \emph{non-explosive}, that is, almost surely, for all $t > 0$, $|\mathscr{T}_{t}| < \infty$. We note that the following lemma does not require Item~\ref{item:finiteness} of Assumption~\ref{ass:gen-ass}. 

\begin{lem} \label{lem:non-expl}
Assume that Items~\ref{item:mutual-independence} and~\ref{item:finiteness} of Assumption~\ref{ass:gen-ass} are satisfied, and for $\alpha > 0$, we have 
\begin{equation} \label{eq:finite-laplace}
\sum_{j=1}^{\infty} \E{e^{-\alpha \sum_{i=1}^{j} X_{i}}} < 1.
\end{equation} 
Then, for any $t >0$, $|\mathscr{T}_{t}| < \infty$. 
\end{lem}

\begin{proof}
For $\alpha > 0$ satisfying $\sum_{j=1}^{\infty} \E{e^{-\alpha \sum_{i=1}^{j} X_{i}}} < 1$, let $Y_{\alpha} \sim \Exp{\alpha}$ be an exponentially distributed random variable, independent of the process $(\mathscr{T}_{t})_{t \geq 0}$. 
Therefore, using the fact that $Y_{\alpha}$ is exponentially distributed, and the independence of the associated random variables, for $v = v_1 \cdots v_{m} \in \mathcal{U}$, we have
\begin{linenomath*}
\begin{align}
\Prob{\mathcal{B}(v) \leq Y_{\alpha}} & = \Prob{\sum_{\ell = 1}
^{m} \sum_{i=1}^{v_{\ell}}  X(v_1 \cdots v_{\ell-1} i) \leq Y_{\alpha}} \\ & = \E{e^{-\alpha \sum_{\ell = 1}
^{m} \sum_{i=1}^{v_{\ell}}  X(v_1 \cdots v_{\ell-1} i)}} = \prod_{\ell = 1}^{m} \E{ e^{-\alpha \sum_{k=1}^{v_{\ell}} X_{k}}}.
\end{align}
\end{linenomath*}
Therefore, 
\begin{linenomath*}
    \begin{align}
        \E{|\mathscr{T}_{Y_{\alpha}}|} = \sum_{v \in \mathcal{U}} \Prob{\mathcal{B}(v) \leq Y_{\alpha}} & = \Prob{\mathcal{B}(\varnothing) \leq Y_{\alpha}} + \sum_{m=1}^{\infty} \sum_{v_{1} = 1}^{\infty} \cdots \sum_{v_{m} = 1}^{\infty} \prod_{\ell = 1}^{m} \E{ e^{-\alpha \sum_{k=1}^{v_{\ell}} X_{k}}} \\ & = 1 + \sum_{m=1}^{\infty} \left(\sum_{j=1}^{\infty} \E{e^{-\alpha \sum_{i=1}^{j} X_{i}}}\right)^{m} 
        \stackrel{\eqref{eq:finite-laplace}}{<} \infty.
    \end{align}
\end{linenomath*}
Thus, $|\mathscr{T}_{Y_{\alpha}}| < \infty$ almost surely. As $\Prob{Y_{\alpha} > t} > 0$, for any $t \geq 0$, this implies that, almost surely, for any $t > 0$, we have $|\mathscr{T}_{t}| < \infty$. 
\end{proof}

 Next, for $u \in \mathcal{U}$, $K \in \mathbb{N}$ we say 
\begin{equation} \label{eq:moderate-defn}
\text{$u$ is $K$-moderate if $u = u_1 \cdots u_m$ with each $u_{i} \leq K$.}
\end{equation}
Moreover, recall that, for a tree $\mathcal{T}$, if $u \notin \mathcal{T}$ we set $\outdeg{u, \mathcal{T}} = -\infty$. Then, with regards to the CMJ process $(\mathscr{T}_{t})_{t \geq 0}$ we define the following event: for $u, v \in \mathcal{U}$ set 
\begin{equation} \label{eq:win-def}
\win(u,v) := \left\{\exists N \in \mathbb{N}\colon \; \forall n \geq N \; \deg(u, \mathscr{T}_{\tau_{n}}) \geq \deg(v, \mathscr{T}_{\tau_{n}})\right\}.
\end{equation}

\begin{lem} \label{lem:finite-persistence-claims}
Assume Assumption~\ref{ass:gen-ass} is satisfied. Then, with regards to the process $(\mathscr{T}_{t})_{t \geq 0}$, we have the following claims:
\begin{enumerate}
    \item \label{item:non-expl-cor} We have $\lim_{n \to \infty} \max_{u \in \mathscr{T}_{\tau_{n}}} \outdeg{u, \mathscr{T}_{\tau_{n}}} = \infty$ almost surely. 
    \item \label{item:winner} If $\sum_{i=1}^{\infty} (X_{i} - X'_{i})$ converges almost surely, then for any $u, v \in \mathcal{U}$ we have 
    \begin{equation} \label{eq:win-one}
    \Prob{\win(u,v) \cup \win(v,u)} = 1.
    \end{equation}
    \item \label{item:no-winner} If $\sum_{i=1}^{\infty} (X_{i} - X'_{i})$ diverges almost surely, there exists an increasing function $\phi\colon \mathbb{N} \rightarrow \mathbb{N}$ such that, for any $u \in \mathcal{U}$, $j \in \mathbb{N}$ we have 
    \begin{equation} \label{eq:win-two}
    \Prob{\exists k \leq \phi(j)\colon \sum_{\ell = 1}^{k} X(uj\ell) < \sum_{\ell = j+1}^{k} X(u\ell)} \geq 1/2.
    \end{equation}
\end{enumerate}
\end{lem}

\begin{rmq}
Note that the proof of Item~\ref{item:winner} of Lemma~\ref{lem:finite-persistence-claims} uses similar ideas to the proof of \cite[Item~1 of Theorem~1.4]{leadership}. {\small\ensymboldremark }
\end{rmq}

\begin{proof}[Proof of Item~\ref{item:non-expl-cor} of Lemma~\ref{lem:finite-persistence-claims}]
First suppose that $\lim_{n \to \infty} \tau_{n} = \infty$ almost surely. For any $v = v_1 \cdots v_{m} \in \mathcal{U}$, with $m \geq 1$ we have
\[
\mathcal{B}(v) = \sum_{\ell = 1}
^{m} \sum_{i=1}^{v_{\ell}}  X(v_1 \cdots v_{\ell-1} i) < \infty \quad \text{almost surely,}
\]
by Item~\ref{item:finiteness} of Assumption~\ref{ass:gen-ass}. These two facts together imply that, for any $v \in \mathcal{U}$, $j \in \mathbb{N}$, there exists a (random) $N \in \mathbb{N}$ such that $\mathcal{B}(vj) \leq \tau_{N}$, hence $vj \in \mathscr{T}_{\tau_{N}}$. In particular, for any $v \in \mathcal{U}$, $\lim_{n \to \infty} \outdeg{v, \mathscr{T}_{\tau_{n}}} = \infty$, which implies the claim. 

Otherwise, we have $\lim_{n \to \infty} \tau_{n} < \infty$ with positive probability. Suppose that, with positive probability we have $\lim_{n \to \infty} \max_{u \in \mathscr{T}_{\tau_{n}}} \outdeg{u, \mathscr{T}_{\tau_{n}}} < \infty$. If 
\begin{equation} \label{eq:pos-max-deg}
\Prob{\lim_{n \to \infty} \max_{u \in \mathscr{T}_{\tau_{n}}} \outdeg{u, \mathscr{T}_{\tau_{n}}} < \infty, \lim_{n \to \infty} \tau_{n} < \infty} > 0,
\end{equation}
it must be the case that, for some $t > 0$, $\Prob{|\mathscr{T}_{t} |= \infty} > 0$, and therefore, by the pigeonhole principle, since the maximal degree is bounded, for some $t > 0$, $K \in \mathbb{N}$
\[
\Prob{\left|u \in \mathscr{T}_{t}\colon u \text{ is $K$-moderate} \right| = \infty} > 0.
\]
This now contradicts the following claim:
\begin{clm} \label{clm:finite-k-moderate}
    Under Assumption~\ref{ass:gen-ass}, almost surely, for all $t \in [0, \infty)$ we have 
    \[
    \left|u \in \mathscr{T}_{t}\colon u \text{ is $K$-moderate} \right| < \infty.
    \]
\end{clm}
We conclude that the right-side of~\eqref{eq:pos-max-deg} is $0$, which implies the result. 
\end{proof}
The proof of Claim~\ref{clm:finite-k-moderate} is similar to~\cite[Proposition~4.4]{inhom-sup-pref-attach}, however we include a proof for brevity and completeness. 
\begin{proof}[Proof of Claim~\ref{clm:finite-k-moderate}]
If $\mathscr{T}^{(K)}_{t}$ denotes the set $\left\{u \in \mathscr{T}_{t}\colon u \text{ is $K$-moderate}\right\}$, one readily verifies that the process $(\mathscr{T}^{(K)}_{t})_{t\geq 0}$ has the same distribution as a CMJ branching process with associated random variables $(X^{(K)}_{j})_{j \in \mathbb{N}}$ satisfying 
\begin{equation}
    X^{(K)}_{j} \sim \begin{cases}
        X_{j}  & \text{if $j \leq K$} \\
        \infty & \text{otherwise.}
    \end{cases}
\end{equation}
Now, for $\alpha > 0$, we have  
\begin{linenomath}
\begin{align} \label{eq:finite-laplace}
\sum_{j=1}^{\infty} \E{e^{-\alpha \sum_{i=1}^{j} X^{(K)}_{i}}} & = \sum_{j=1}^{K} \E{e^{-\alpha \sum_{i=1}^{j} X_{i}}} \\ & = \sum_{j=1}^{K} \left( \prod_{i=1}^{j} \Prob{X_{i} = 0} + \E{e^{-\alpha \sum_{i=1}^{j} X_{i}} \mathbf{1}_{\left\{\sum_{i=1}^{j} X_{i} > 1 \right\}}}\right). 
\end{align} 
\end{linenomath}
By Items~\ref{item:mutual-independence} and~\ref{item:finiteness} of Assumption~\ref{ass:gen-ass},  for $\alpha$ sufficiently large, the right-side of~\eqref{eq:finite-laplace} is smaller than one. Therefore, by Lemma~\ref{lem:non-expl}, the process $(\mathscr{T}^{(K)}_{t})_{t \geq 0}$ is non-explosive. In other words, almost surely, for all $t \in [0, \infty)$
	\begin{equation} \label{eq:non-exposive-L}
		\left|\mathscr{T}^{(K)}_{t}\right| < \infty.
	\end{equation}
 This implies the claim.
\end{proof}
We continue the proof of Lemma~\ref{lem:finite-persistence-claims}. 
\begin{proof}[Proof of Item~\ref{item:winner} of Lemma~\ref{lem:finite-persistence-claims}]
First note that having either  \[\lim_{n \to \infty} \outdeg{u, \mathscr{T}_{\tau_{n}}} < \infty \quad \text{ or } \quad \lim_{n \to \infty} \outdeg{v, \mathscr{T}_{\tau_{n}}} < \infty\] already implies that $\win(u,v) \cup \win(v,u)$ occurs. Therefore, we need only show $\win(u,v) \cup \win(v,u)$ occurs on the event that $\lim_{n \to \infty} \outdeg{u, \mathscr{T}_{\tau_{n}}} = \infty$ and $\lim_{n \to \infty} \outdeg{v, \mathscr{T}_{\tau_{n}}} = \infty$. 
Now, by the assumption that $\sum_{j=1}^{\infty} (X'_{j} - X_{j})$ converges almost surely, for $u, v \in \mathcal{U}$ we have
    \[
    \left|\mathcal{B}(u) - \mathcal{B}(v) + \sum_{j=1}^{\infty} \left(X(uj) - X(vj)\right)\right| < \infty \quad \text{ almost surely.} 
    \]
    Assume, without loss of generality, that $|u| \leq |v|$. Let $N$ be chosen such that $\mathcal{B}(v)$ is independent of $(X(u n))_{n \geq N}$. (For example, if $u$ is a non-parent ancestor of $v$, we may choose the value $N$ such that $u(N-1)$ is the ancestor of $v$, whilst if $v = uj$ for $j \in \mathbb{N}$, we can choose $N > j$.) Then, the sequences $(X(u n))_{n \geq N}$ and $(X(v n))_{n \geq N}$ are independent.
    We now have two cases: 
    \begin{enumerate}[label = (\Roman*)]
            \item First suppose that the sequence $(S_{j})_{j \in \mathbb{N}}$ defined by $S_{j} := X_{j} - X'_{j}$ does not satisfy Equation~\eqref{eq:atomic-sum}. Then, by Lemma~\ref{lem:no-atom}, the random variable 
            \[
                \sum_{j=N}^{\infty} (X(uj) - X(vj))
            \]
            contains no atom on $\mathbb{R}$. Hence, as the summands are independent
            \[
            \Prob{\mathcal{B}(u) - \mathcal{B}(v) + \sum_{j=1}^{\infty} \left(X(uj) - X(vj)\right) = 0} = 0.\]
            Therefore, almost surely, there exists $K_{0} \in \mathbb{N}$ such that for all $k \geq K_0$ either 
            \begin{equation} \label{eq:eventual-dominance}
            \mathcal{B}(u) - \mathcal{B}(v) + \sum_{j=1}^{k} \left(X(uj) - X(vj)\right) > 0 \quad \text{ or } \quad \mathcal{B}(u) - \mathcal{B}(v) + \sum_{j=1}^{k} \left(X(uj) - X(vj)\right) < 0. 
            \end{equation}
            In other words, for all $k \geq K_{0}$, we have 
            $\mathcal{B}(uk) > \mathcal{B}(vk)$ or $\mathcal{B}(uk) < \mathcal{B}(vk)$. This in turn implies that for all $k$ sufficiently large $u$ reaches out-degree $k$ in $(\mathscr{T}_{\tau_{n}})_{n \in \mathbb{N}}$ before $v$ reaches out-degree $k$ in $(\mathscr{T}_{\tau_{n}})_{n \in \mathbb{N}}$ or vice-versa. Thus $\win(u,v) \cup \win(v,u)$ occurs.
            \item Otherwise, the sequence $(S_{j})_{j \in \mathbb{N}}$ defined by $S_{j} := X_{j} - X'_{j}$ does satisfy Equation~\eqref{eq:atomic-sum}. In particular, this implies that $\sum_{j=1}^{\infty} \Prob{X(uj) \neq X(vj)} < \infty$, thus by the Borel--Cantelli lemma, almost surely, $X(uj) = X(vj)$ for all but finitely many $j \in \mathbb{N}$. Thus, for some $K_0 \in \mathbb{N}$, it is either the case that for all $k \geq K_0$~\eqref{eq:eventual-dominance} is satisfied, or for all $k \geq K_0$
            \[
            \mathcal{B}(u) - \mathcal{B}(v) + \sum_{j=1}^{k} \left(X(uj) - X(vj)\right) = 0. 
            \]
            In particular, for all $k \geq K_0$, $\mathcal{B}(uk) \geq \mathcal{B}(vk)$ or $\mathcal{B}(uk) \leq \mathcal{B}(vk)$,
            which again implies that $\win(u,v) \cup \win(v,u)$ occurs.
            \end{enumerate}
\end{proof}

\begin{proof}[Proof of Item~\ref{item:no-winner} of Lemma~\ref{lem:finite-persistence-claims}]
Note that, by Assumption~\ref{ass:gen-ass}, for any $j\in \mathbb{N}$, the values of $(X(u\ell) - X(uj\ell))_{\ell \geq j+1}$ are symmetric and independent. If, by assumption, $\sum_{i=1}^{\infty} (X_{i} - X'_{i})$ diverges almost surely, Lemma~\ref{lem:fluctuating-sum} implies that 
\[
\limsup_{n \to \infty} \sum_{\ell=j+1}^{n} (X(u\ell) - X(uj\ell)) = \infty \quad \text{almost surely.}
\]
Since Assumption~\ref{ass:gen-ass} also implies that $\sum_{\ell=1}^{j} X(u\ell) < \infty$ almost surely, we therefore have 
\[
\Prob{\exists k \in \mathbb{N}\colon \sum_{\ell = 1}^{k} X(uj\ell) < \sum_{\ell=j+1}^{k} X(u\ell)} = 1. 
\]
By monotone convergence, we can write the left hand side as \[\lim_{n \to \infty} \Prob{\exists k \leq n\colon \sum_{\ell = 1}^{k} X(uj\ell) < \sum_{\ell=j+1}^{k} X(u\ell)},\] and therefore, may define $\phi(j)$ such that
\[
\phi(j) := \inf\left\{n > j\colon \Prob{\exists k \leq n\colon \sum_{\ell = 1}^{k} X(uj\ell) < \sum_{\ell=j+1}^{k} X(u\ell)} \geq 1/2\right\}.
\]
\end{proof}

\subsubsection{Proof of Item~\ref{item:non-persistence} of Theorem~\ref{thm:persistence}} \label{sec:non-persistence}

\begin{proof}[Proof of Item~\ref{item:non-persistence} of Theorem~\ref{thm:persistence}]
Suppose that $\sum_{i=1}^{\infty} (X_{i} - X'_{i})$ converges almost surely, and assume, in order to obtain a contradiction, that  $(\mathscr{T}_{\tau_{n}})_{n \in \mathbb{N}}$ contains a persistent hub with positive probability. Then, by Item~\ref{item:non-expl-cor} of Lemma~\ref{lem:finite-persistence-claims}, for any persistent hub $u$, say, we have 
\begin{equation} \label{eq:degree-to-infinity} 
\lim_{n \to \infty} \outdeg{u, \mathscr{T}_{\tau_{n}}} = \infty.
\end{equation}
 On the other hand, for any $u \in \mathcal{U}$, $j \in \mathbb{N}$, with $\phi$ as defined in Lemma~\ref{lem:finite-persistence-claims}, we have 
\begin{linenomath}
\begin{align} \label{eq:inf-many-j-catch-up}
&\Prob{\exists k \leq \phi(j)\colon \mathcal{B}(uj) + \sum_{\ell = 1}^{k} X(uj\ell) < \mathcal{B}(u) + \sum_{\ell = 1}^{k} X(u\ell)} \\ & \hspace{3cm} = \Prob{\exists k \leq \phi(j)\colon \sum_{\ell = 1}^{k} X(uj\ell) < \sum_{\ell = j+1}^{k} X(u\ell)}. 
\end{align}
\end{linenomath}
As $\phi\colon \mathbb{N} \rightarrow \mathbb{N}$ is strictly increasing, the right-side of~\eqref{eq:inf-many-j-catch-up} involves independent events. Therefore, by Lemma~\ref{lem:finite-persistence-claims} and the Borel-Cantelli lemma, we have
\[ 
\Prob{\exists \text{ infinitely many } j \in \mathbb{N}, k \leq \phi(j)\colon \sum_{\ell = 1}^{k} X(uj\ell) < \sum_{\ell = j+1}^{k} X(u\ell)} = 1.
\]
Combining this with Equation~\eqref{eq:degree-to-infinity}, it must be the case that infinitely many $uj$ have out-degrees that `overtake' the out-degree of $u$ with probability $1$. Thus, for any $u \in \mathcal{U}$, $\Prob{u \text{ is a persistent hub}} = 0$. But then, taking a countable intersection
\[
\Prob{\exists u \in \mathcal{U}\colon u \text{ is a persistent hub}} = 0, 
\]
which contradicts the assumption that with positive probability $(\mathscr{T}_{\tau_{n}})_{n \in \mathbb{N}}$ contains a persistent hub. 
\end{proof}

\subsubsection{Proof of Item~\ref{item:persistent-hub} of Theorem~\ref{thm:persistence}} \label{sec:persistent-hub}

The proof of Item~\ref{item:persistent-hub} of Theorem~\ref{thm:persistence} uses Item~\ref{item:non-persistence} of Theorem~\ref{thm:persistence}, and Lemma~\ref{lem:finite-catch-up} below. We first state this lemma, prove  Item~\ref{item:persistent-hub} of Theorem~\ref{thm:persistence}, and then prove the lemma over the rest of the section.

Lemma~\ref{lem:finite-catch-up} relates to the number of individuals in $(\mathscr{T}_{t})_{t \geq 0}$ that `catch up' to each of their ancestors in out-degree. More precisely,  for $u, v \in \mathcal{U}$, with $v = v_1 \cdots v_m$, we define the event 
\[
\left\{uv \text{ catches up to } u\right\} := \left\{\exists j \colon \mathcal{B}(u v (v_1 + j)) \leq \mathcal{B}(u(v_1 + j)) \right\}.
\]
 In other words, $uv$ catches up to $u$ if, at some point, $uv$ produces $v_1 + j$ children before $u$ does, so that the out-degree of $uv$ in $(\mathscr{T}_{t})_{t \geq 0}$ `catches up' to the out-degree of $u$. 

\begin{lem} \label{lem:finite-catch-up}
Suppose that Equation~\eqref{eq:persistent} is satisfied. Then, in the process $(\mathscr{T}_{t})_{t \geq 0}$, the set 
    \begin{equation} \label{eq:catch-up-set}
    \mathcal{P} := \left\{u \in \mathcal{U}\colon u \text{ catches up to each of its ancestors}\right\}
    \end{equation}
    is finite almost surely.
\end{lem} 

\begin{proof}[Proof of Item~\ref{item:persistent-hub} of Theorem~\ref{thm:persistence}]
First note that if the series $\sum_{i=1}^{\infty} (X_{i} - X'_{i})$ diverges almost surely, Equation~\eqref{eq:persistent} cannot be satisfied.
Indeed suppose otherwise. Then, by Lemma~\ref{lem:overtake-bound}, with $Y$ given by a deterministic constant $C > 0$ and $\lambda = \alpha$
\begin{linenomath}
\begin{align}
1 = \Prob{\limsup_{N \to \infty} \sum_{i=K+1}^{N}(X_{i} - X'_{i}) > C} & = \Prob{\exists N \in \mathbb{N}\colon \sum_{i=K+1}^{N}(X_{i} - X'_{i}) > C} 
\\ & \leq \left(\prod_{i=K+1}^{\infty} \E{e^{\alpha(X_{i} - X'_{i})}} \right) e^{-C},
\end{align}
\end{linenomath}
where the first equality follows from Lemma~\ref{lem:fluctuating-sum}. We may choose $C$ sufficiently large that the right side is smaller than one, thus obtaining a contradiction. 

Hence, Equation~\eqref{eq:persistent} implies that $\sum_{i=1}^{\infty} (X_{i} - X'_{i})$ converges almost surely. Therefore, by Equation~\eqref{eq:win-one} in Lemma~\ref{lem:finite-persistence-claims}, with $\win(u,v)$ as defined in~\eqref{eq:win-def}, as an intersection over a countable set, we have 
\begin{equation} \label{eq:finite-leadership}
\Prob{\bigcap_{S \subseteq \mathcal{U}\colon |S| < \infty}  \bigcap_{u,v \in S} \left\{\win(u,v) \cup \win(v,u) \right\}} = 1.  
\end{equation}
Note that, for any finite set $S$
\begin{linenomath}
\begin{align} \label{eq:extra}
& \bigcap_{u,v \in S} \left\{\win(u,v) \cup \win(v,u) \right\} \\ & \hspace{3cm}  = \left\{\exists u \in S, N \in \mathbb{N}\colon \; \forall n \geq N \; \max_{v \in S} \outdeg{v, \mathscr{T}_{\tau_{n}}} = \outdeg{u, \mathscr{T}_{\tau_{n}}}\right\}. 
\end{align}
\end{linenomath}
By assumption, $\mathcal{P}$ as defined in~\eqref{eq:catch-up-set} is finite almost surely. 
Therefore, Equations~\eqref{eq:finite-leadership} and~\eqref{eq:extra} imply that
\[
\Prob{\exists N \in \mathbb{N}, u \in \mathcal{P}: \; \forall n \geq N \; \max_{v \in \mathcal{P}} \outdeg{v, \mathscr{T}_{\tau_{n}}} = \outdeg{u, \mathscr{T}_{\tau_{n}}}} = 1. 
\] 
As any node of maximal out-degree must be a node that catches up to each of its ancestors, this implies that $(\mathscr{T}_{\tau_{n}})_{n \in \mathbb{N}_{0}}$ contains a persistent hub. 
\end{proof}
The rest of this section is dedicated to the proof of Lemma~\ref{lem:finite-catch-up}. First, we have the following. 
\begin{prop} \label{prop:late-catch-up}
   Suppose that Equation~\eqref{eq:persistent} is satisfied for $\alpha > 0$ and $K \in \mathbb{N}$. Then, for any $u \in \mathcal{U}$,  
    \begin{equation}
       \E{\left| \left\{v = v_1 \cdots v_{m} \in \mathcal{U} \colon v_1 \geq K, uv\text{ catches up to } u \right\}\right| }= \sum_{v\in \mathcal{U}\colon v_1 \geq K} \Prob{uv \text{ catches up to } u} < \infty. 
    \end{equation}
\end{prop}

\begin{proof}
We choose $K \in \mathbb{N}$ and $\alpha > 0$ so that Equation~\eqref{eq:persistent} is satisfied. 
By applying Lemma~\ref{lem:overtake-bound}, with quantities chosen such that
\[\lambda = \alpha, \; Y = \sum_{\ell = 1}
^{m-1} \sum_{i=1}^{v_{\ell+1}}  X(u v_1 \cdots v_{\ell} i), \; (X_{i})_{i \in \mathbb{N}} = (X(u v_1 \cdots v_m k))_{k \in \mathbb{N}}, \text{ and }(X'_{i})_{i \in \mathbb{N}} = (X(u k))_{k \in \mathbb{N}},\] for $v = v_1 \cdots v_m \in \mathcal{U}$ such that $v_1 \geq K$, we have 
\begin{linenomath*}
\begin{align} \label{eq:applied-overtake}
\Prob{uv \text{ catches up to } u} & = \Prob{\exists j \colon \mathcal{B}(u v (v_1 + j)) \leq \mathcal{B}(u(v_1 + j))} 
\\ &  = \Prob{\exists j \colon \sum_{\ell = 1}
^{m-1} \sum_{i=1}^{v_{\ell+1}}  X(u v_1 \cdots v_{\ell} i)  + \sum_{k=1}^{v_1 + j} X(u v_1 \cdots v_{m} k) \leq  \sum_{k=v_1+1}^{v_1 + j} X(u k)}
\\ &  \stackrel{\eqref{eq:maximal-bound}}{\leq} \left(\prod_{i=K+1}^{\infty}\E{e^{\alpha (X'_{i} - X_{i})}} \right) \left(\prod_{\ell=1}^{m-1} \E{e^{-\alpha \sum_{k=1}^{v_{\ell+1}}X_{k}}}\right) \E{e^{-\alpha \sum_{k=1}^{v_1} X_k}}  
\\ & \hspace{0.15cm}= \left(\prod_{i=K+1}^{\infty}\E{e^{\alpha (X'_{i} - X_{i})}} \right) \left(\prod_{\ell=1}^{m} \E{e^{-\alpha \sum_{k=1}^{v_{\ell}}X_{k}}}\right). 
\end{align}
\end{linenomath*}
By summing over the possible values of $v \in \mathcal{U}$ we have 
\begin{linenomath*}
    \begin{align} \label{eq:catch-up-comp}
     &   \sum_{v\in \mathcal{U}\colon v_1 \geq K} \Prob{uv \text{ catches up to } u} = \sum_{m=1}^{\infty} \sum_{v \in \mathcal{U}, v_1 \geq K, |v| = m} \Prob{uv_{1} \cdots v_{m} \text{ catches up to } u}
        \\ & \hspace{1cm} \stackrel{\eqref{eq:applied-overtake}}{\leq} \left(\prod_{i=K+1}^{\infty}\E{e^{\alpha (X'_{i} - X_{i})}} \right) \sum_{m=1}^{\infty} \sum_{v_1 = K}^{\infty} \sum_{v_2 = 1}^{\infty} \cdots \sum_{v_{m} =1}^{\infty} 
    \left(\prod_{\ell = 1}^{m} \E{e^{-\alpha \sum_{k=1}^{v_{\ell}} X_{k}}}\right) 
\\ & \hspace{1.15cm} = \left(\prod_{i=K+1}^{\infty}\E{e^{\alpha (X'_{i} - X_{i})}} \right) \sum_{m=1}^{\infty} \sum_{v_1 = K}^{\infty} \E{e^{-\alpha\sum_{k=1}^{v_{1}} X_{k}}} \left(\sum_{s=1}^{\infty} \E{e^{-\alpha \sum_{k=1}^{s} X_{k}}}\right)^{m-1}
\\ & \hspace{1.15cm} < \left(\prod_{i=K+1}^{\infty}\E{e^{\alpha (X'_{i} - X_{i})}} \right) \sum_{m=1}^{\infty} \left(\sum_{s=1}^{\infty} \E{e^{-\alpha \sum_{k=1}^{s} X_{k}}}\right)^{m} 
\stackrel{\eqref{eq:persistent}}{<} \infty,
    \end{align}
\end{linenomath*}
where the last line follows from the fact that, since $\sum_{j=1}^{\infty} \E{e^{-\alpha \sum_{i=1}^{j} X_{i}}} < 1$ by Equation~\eqref{eq:persistent}, the geometric series converges. 
\end{proof}


We are finally ready to prove Lemma~\ref{lem:finite-catch-up}:
\begin{proof}[Proof of Lemma~\ref{lem:finite-catch-up}]
As in Proposition~\ref{prop:late-catch-up}, we choose $K \in \mathbb{N}$ and $\alpha > 0$ so that Equation~\eqref{eq:persistent} is satisfied. Recalling the definition from~\eqref{eq:moderate-defn}, we denote by \[\mathcal{U}_{K-1}:= \left\{u \in \mathcal{U}\colon u \text{ is $(K-1)$-moderate} \right\};\] i.e, the subset of $\mathcal{U}$ consisting of $(K-1)$-moderate individuals. We first condition on the sigma algebra $\mathcal{F}_{\mathcal{B}(K)}$ generated by the process up until the birth of $K$, the $K$th child of $\varnothing$. 
Now, since $\mathcal{B}(K) < \infty$ almost surely, by Lemma~\ref{lem:non-expl}, $|\mathscr{T}_{\mathcal{B}(K)}| < \infty$ almost surely. We set $\tilde{m}$ to be the maximum length of a $(K-1)$-moderate individual born before $K$, i.e.,
\[
\tilde{m} := \sup \left\{|u|\colon u \in \mathscr{T}_{\mathcal{B}(K)} \cap \mathcal{U}_{K-1}\right\}. 
\]
Then, define the random, $\mathcal{F}_{\mathcal{B}(K)}$-measurable sets
\[
A^{+} := \left\{u \in \mathcal{U}_{K-1}\colon |u| = \tilde{m} +1 \right\},  \quad \text{and} \quad A^{-} := \left\{u \in \mathcal{U}_{K-1}\colon |u| < \tilde{m} +1\right\}. 
\]
We use these sets to complete the proof of the claim. An informal overview of the argument is as follows. First, we argue that $A^{+}$ and $A^{-}$ are finite almost surely. Next, $A^{+}$ represents the maximal `boundary' of $(K-1)$-moderate nodes born before $K$. As the root $\varnothing$ is already `large', since $K$ has already been born, by exploiting a similar argument to Proposition~\ref{prop:late-catch-up}, we expect that each node in $A^{+}$ to have only finitely many descendants that catch up to $\varnothing$. Otherwise, a node either belongs to $A^{-}$, or is of the form $u = ab$, with $a \in A^{-}$ and $b= b_1 \cdots b_{\ell}$, with $b_{1} \geq K$. For the latter, we can exploit Proposition~\ref{prop:late-catch-up} to show that, for each $a \in A^{-}$, there exists only finitely many such individuals $ab$. 

In formalising this argument, it is helpful to define the following sets. For a given $u \in \mathcal{U}$ set 
\[
\mathcal{C}_{u} := \left\{uv\colon v = v_1 \cdots v_{m} \in \mathcal{U},  v_1 \geq K, uv \text{ catches up to } u \right\}
\]
and 
\[
\mathcal{D}_{u} := \left\{uv \in \mathcal{U}\colon |v| \geq 1, uv \text{ catches up to } \varnothing \right\}.
\]
Now, note that, with $\mathcal{P}$ as defined by~\eqref{eq:catch-up-set} we may write
\begin{equation} \label{eq:covering-catch-up}
    \mathcal{P} \subseteq \mathcal{C}_{\varnothing} \cup A^{-} \cup A^{+} \cup \left(\bigcup_{u \in A^{-}} \mathcal{C}_{u} \right) \cup \left(\bigcup_{a \in A^{+}} \mathcal{D}_{u} \right). 
\end{equation}
Indeed, $C_{\varnothing}$ covers the individuals $u$ with $u_{1} \geq K$ that catch up to $\varnothing$, whilst $\left(\bigcup_{u \in A^{-}} \mathcal{C}_{u} \right) \cup A^{-}$ and $\bigcup_{u \in A^{+}} \mathcal{D}_{u} $ cover descendants of $A^{-}$ and $A^{+}$ respectively, that catch up to all of their ancestors. 

We show that each of the sets on the right-hand side of~\eqref{eq:covering-catch-up} are finite almost surely, hence so is $\mathcal{P}$. 
\begin{enumerate}[label = (\Roman*)]
\item Proposition~\ref{prop:late-catch-up} showed that $\E{|\mathcal{C}_{\varnothing}|} < \infty$, hence $|\mathcal{C}_{\varnothing}|$ is finite almost surely. 
\item We know that $|A^{-}| = ((K-1)^{\tilde{m} + 1} - 1)/(K-2)$  and $|A^{+}| = (K-1)^{\tilde{m} +1}$. By Lemma~\ref{lem:non-expl}, which shows that $\tilde{m} < \infty$ almost surely, these are both finite almost surely. 
\item Since $|A^{-}| < \infty$ almost surely, we have  
\begin{linenomath}
\begin{align}
    \Prob{\bigg|\bigcup_{u \in A^{-}} \mathcal{C}_{u} \bigg| = \infty \, \bigg | \, \mathcal{F}_{\mathcal{B}(K)}} & = \Prob{\exists u \in A^{-}\colon \,  |\mathcal{C}_{u}| = \infty \, \bigg | \, \mathcal{F}_{\mathcal{B}(K)}}
    \\ & \leq \sum_{u \in A^{-}} \Prob{\left|\mathcal{C}_{u}\right| = \infty \, \big | \, \mathcal{F}_{\mathcal{B}(K)}} = 0, 
\end{align}
\end{linenomath}
almost surely, where the last equality follows from Proposition~\ref{prop:late-catch-up}. 
\item Finally, for $\left|\bigcup_{a \in A^{+}} \mathcal{D}_{a}\right|$, first note that by definition, for each $a \in A^{+}$, we have $\mathcal{B}(a) - \mathcal{B}(K) > 0$. Therefore, for each $a \in A^{+}$, we have, 
\end{enumerate}
\begin{linenomath*}
    \begin{align*}
        & \E{\left|\mathcal{D}_{a}\right|\, \big | \, \mathcal{F}_{\mathcal{B}(K)}} = \sum_{b \in \mathcal{U}} \Prob{ab\text{ catches up to } \varnothing \, \big | \, \mathcal{F}_{\mathcal{B}(K)}}
        \\ & = \sum_{b \in \mathcal{U}} \Prob{\exists j \colon \mathcal{B}(a) - \mathcal{B}(K) + \sum_{\ell = 0}
^{m-1} \sum_{i=1}^{b_{\ell+1}}  X(a b_1 \cdots b_{\ell} i)  + \sum_{k=1}^{K + 1+ j} X(a b_1 \cdots b_{m} k) \leq \sum_{k=K+1}^{K+1 + j} X(k) \, \bigg | \, \mathcal{F}_{\mathcal{B}(K)}}
\\ & \hspace{1cm} \leq \sum_{b \in \mathcal{U}} \Prob{\exists j \colon \sum_{\ell = 0}
^{m-1} \sum_{i=1}^{b_{\ell+1}}  X(a b_1 \cdots b_{\ell} i)  + \sum_{k=1}^{K + 1+ j} X(a b_1 \cdots b_{m} k) \leq \sum_{k=K+1}^{K+1 + j} X(k) \, \bigg | \, \mathcal{F}_{\mathcal{B}(K)}}
\\ & \hspace{1cm} = \sum_{b \in \mathcal{U}} \Prob{\exists j \colon \sum_{\ell = 0}
^{m-1} \sum_{i=1}^{b_{\ell+1}}  X(a b_1 \cdots b_{\ell} i)  + \sum_{k=1}^{K +1 + j} X(a b_1 \cdots b_{m} k) \leq \sum_{k=K+1}^{K+1 + j} X(k)} \quad \text{almost surely},
    \end{align*}
\end{linenomath*}
where the last equality follows from the fact that the random variables concerned are independent of $\mathcal{F}_{\mathcal{B}(K)}$. Now, by applying Lemma~\ref{lem:overtake-bound} in a similar manner to its usage in Equation~\eqref{eq:applied-overtake} from Proposition~\ref{prop:late-catch-up}, we may bound the previous above, so that, for $a \in A^{+}$

\begin{linenomath*}
\begin{align*}
\E{\left|\mathcal{D}_{a}\right| \, \big | \, \mathcal{F}_{\mathcal{B}(K)}} & \leq \left(\prod_{i=K+1}^{\infty}\E{e^{\alpha (X'_{i} - X_{i})}} \right) \sum_{m=1}^{\infty} \sum_{b_1 = 1}^{\infty} \cdots \sum_{b_{m} = 1}^{\infty} \left(\prod_{\ell=1}^{m} \E{e^{-\alpha \sum_{j=1}^{b_{\ell}} X_{j}}}\right) \E{e^{-\alpha \sum_{j=1}^{K} X_{j}}} \\ & = \left(\prod_{i=K+1}^{\infty}\E{e^{\alpha (X'_{i} - X_{i})}} \right)  \E{e^{-\alpha \sum_{j=1}^{K} X_{j}}} \sum_{m=1}^{\infty} \left(\sum_{j=1}^{\infty} \E{e^{-\alpha \sum_{i=1}^{j} X_{i}}}\right)^{m} \stackrel{\eqref{eq:persistent}}{<} \infty.
\end{align*}
\end{linenomath*}
Thus, summing the geometric series in the above display, and using the almost sure finiteness of $A^{+}$ we have 
\begin{linenomath*}
\begin{align*}
&        \E{\bigg|\bigcup_{a \in A^{+}} \mathcal{D}_{a}\bigg| \, \bigg | \, \mathcal{F}_{\mathcal{B}(K)}} 
< \infty, \quad \text{almost surely.}
\end{align*}
\end{linenomath*}
\end{proof}

\subsubsection{Proof of Theorem~\ref{thm:unique}} \label{sec:unique}
\begin{proof}[Proof of Theorem~\ref{thm:unique}] 
We first have the following claim:
\begin{clm} \label{clm:inf-many-equal}
    Under the assumptions of Theorem~\ref{thm:unique}, for any $u, v \in \mathbb{N}$, we have 
    \[
    \Prob{\deg(u, \mathscr{T}_{\tau_{n}}) = \deg(v, \mathscr{T}_{\tau_{n}}) \text{ for infinitely many } n\in \mathbb{N}} = 0. 
    \]
\end{clm}
By Claim~\ref{clm:inf-many-equal}, and taking a union bound over the countably many pairs $u,v \in \mathcal{U}$, we have
\begin{equation} \label{eq:equal-inf-often}
\Prob{\exists u, v \in \mathcal{U}\colon \deg(u, \mathscr{T}_{\tau_{n}}) = \deg(v, \mathscr{T}_{\tau_{n}}) \text{ for infinitely many } n\in \mathbb{N}} = 0. 
\end{equation}
Now, any candidate persistent hub must `catch up' in out-degree of each of its ancestors, hence, by Lemma~\ref{lem:finite-catch-up}, the set  
\[
\mathcal{P} := \left\{u \in \mathcal{U}\colon \exists n \in \mathbb{N}_{0} \text{ such that } \deg(u, \mathcal{T}_{n}) = \max_{v \in \mathcal{T}_{n}} \deg(v, \mathcal{T}_{n}) \right\}
\]
is finite, almost surely. Suppose that, for $u, v \in \mathcal{P}$ we have $\deg(u, \mathscr{T}_{\tau_{n}}) = \max_{w \in \mathscr{T}_{\tau_{n}}} \deg(w, \mathscr{T}_{\tau_{n}})$ for infinitely many $n$, and $\deg(v, \mathscr{T}_{\tau_{n}}) = \max_{w \in \mathscr{T}_{\tau_{n}}} \deg(w, \mathscr{T}_{\tau_{n}})$ for infinitely many $n$. As the out-degree of $u$ needs to `catch up' to the out-degree of $v$ for $u$ to become an index of maximal out-degree, it must be the case that $\deg(u, \mathscr{T}_{\tau_{n}}) = \deg(v, \mathscr{T}_{\tau_{n}})$ for infinitely many $n \in \mathbb{N}$. By Equation~\eqref{eq:equal-inf-often}, it must be the case that $u=v$ almost surely, proving uniqueness.  
\end{proof}
\begin{proof}[Proof of Claim~\ref{clm:inf-many-equal}]
Suppose that for $u, v \in \mathcal{U}$, we have $\deg(u, \mathscr{T}_{\tau_{n}}) = \deg(v, \mathscr{T}_{\tau_{n}})$ for infinitely many $n \in \mathbb{N}$ with positive probability. Then, it must be the case that, for infinitely many $j \in \mathbb{N}$,
\[
\mathcal{B}(u) + \sum_{i=1}^{j} X(ui) \leq \mathcal{B}(v) + \sum_{i=1}^{j} X(vi) < \mathcal{B}(u) + \sum_{i=1}^{j+1} X(ui), 
\]
which implies that
\begin{equation} \label{eq:infinitely-many-j}
0 \leq \mathcal{B}(u) - \mathcal{B}(v) + \sum_{i=1}^{j} (X(uj) - X(vj)) < X(u(j+1)) \quad \text{for infinitely many $j$}. 
\end{equation}
Now, 
\begin{enumerate}[label = (\Roman*)]
\item If Item~\ref{item:borel-cantelli-ass} of Theorem~\ref{thm:unique} is satisfied, by the Borel-Cantelli lemma, we almost surely have $0 \leq \mathcal{B}(u) - \mathcal{B}(v) + \sum_{i=1}^{j} (X(ui) - X(vi)) < X(u(j+1))$ for only finitely many $j$, thus Equation~\eqref{eq:infinitely-many-j} cannot be satisfied with positive probability. 
\item Otherwise, if Item~\ref{item:conv-series-and-borel} of Theorem~\ref{thm:unique} is satisfied, first, by Lemma~\ref{lem:no-atom}, the almost surely convergent series  $\sum_{i=1}^{\infty} (X(ui) - X(vi))$ contains no atom on $\mathbb{R}$. Hence, the random variable
\[\mathcal{B}(u) - \mathcal{B}(v) + \sum_{i=1}^{\infty} (X(ui) - X(vi))
\]
contains no atom at $0$. On the other hand, by the Borel-Cantelli lemma, almost surely, for any $\eps > 0$, we have $X_{j} \leq \eps$ for all but finitely many $j$. 
Now, suppose that~\eqref{eq:infinitely-many-j} is satisfied. Then, for any $\eps > 0$ 
\[
0 \leq \mathcal{B}(u) - \mathcal{B}(v) + \sum_{i=1}^{\infty} (X(ui) - X(vi)) < \eps. 
\]
This can only be the case if $\mathcal{B}(u) - \mathcal{B}(v) + \sum_{i=1}^{\infty} (X(ui) - X(vi))$ contains an atom at $0$, a contradiction. 
\end{enumerate}
\end{proof}

\subsection{Proofs of Theorems~\ref{thm:main} and~\ref{thm:superlinear}} \label{sec:pref-attach-proofs}
For the proof of Theorem~\ref{thm:main}, first suppose that, given the random sequence $(F(k))_{k \in \mathbb{N}_0}$, we define the sequence $(X_{i})_{i \in \mathbb{N}_0}$ such that, 
\begin{equation} \label{eq:waiting-time-cont-embed}
X_{i} \sim Y_{i} \quad \text{where, conditional on $F(i-1)$, we have} \quad Y_{i} \sim \Exp{F(i-1)}.
\end{equation}
In other words, each $X_{i}$ is defined as a mixture of exponential random variables, with rate parameter $F(i-1)$.
Define the process $(\mathscr{T}_{t})_{t \geq 0}$ with this choice of $(X_{i})_{i \in \mathbb{N}_0}$. Then, applying a standard embedding procedure that exploits the memoryless property of the exponential distribution and properties of minima of independent exponential random variables (see, for example, \cite[Section~2.1]{rec-trees-fit}), we have, up to re-labelling of nodes, 
\[(\mathscr{T}_{\tau_{n}})_{n \in \mathbb{N}_{0}} \sim (\mathcal{T}_{n})_{n \in \mathbb{N}_0}.\] 
It thus suffices to apply Theorems~\ref{thm:persistence} and~\ref{thm:unique} to the process $(\mathscr{T}_{\tau_{n}})_{n \in \mathbb{N}_{0}}$, with $(X_{i})_{i \in \mathbb{N}_0}$ as defined in~\eqref{eq:waiting-time-cont-embed}.   

We start with the proofs of Items~\ref{item:non-preference} and~\ref{item:preference-one} of Theorem~\ref{thm:persistence}, which are straightforward applications of the more general Theorems~\ref{thm:persistence} and Theorem~\ref{thm:unique}. 

\subsubsection{Proof of Theorem~\ref{thm:main}}
\begin{proof}[Proof of Theorem~\ref{thm:main}]
First note that, with $(X_{i})_{i \in \mathbb{N}}$ as defined in~\eqref{eq:waiting-time-cont-embed}, Assumption~\ref{ass:gen-ass} is satisfied: mutual independence of the $(X_{i})_{i\in \mathbb{N}}$ follows from the independence of $(F(i))_{i \in \mathbb{N}_{0}}$, whilst, by the properties of the exponential distribution, for any $j \in \mathbb{N}$, we have $0< X_{j} < \infty$ almost surely, which implies the second and third criteria. 

Now, by \cite[Item~1 of Theorem~1.4 \& Item~2 of Theorem~1.10]{leadership} \footnote{Note that `strict leadership' in that model implies `leadership'.}, if $(X_{i})_{i \in \mathbb{N}}$ and $(X'_{i})_{i \in \mathbb{N}}$ are independent collections of independent random variables distributed according to~\eqref{eq:waiting-time-cont-embed}, we have 
\[
\sum_{i=1}^{\infty} (X_{i} - X'_{i}) \quad \text{converges almost surely, if and only if} \quad \sum_{i=0}^{\infty} \frac{1}{F(i)^2} < \infty \quad \text{almost surely.}
\]
Thus, Item~\ref{item:non-preference} of Theorem~\ref{thm:main} already follows from Item~\ref{item:non-persistence} of Theorem~\ref{thm:persistence}. 

For Item~\ref{item:preference-one} of Theorem~\ref{thm:main}, suppose that $\lambda$ satisfies~\eqref{eq:lambda-eq}. Note also, that the quantity appearing in~\eqref{eq:lambda-eq} is decreasing in $\lambda$. Therefore, using the formula for the Laplace transform of an exponential random variable, and the independence of $(F(i))_{i \in \mathbb{N}_{0}}$, for some $\alpha \geq \lambda$ we have 
\begin{equation} \label{eq:alpha-candidate}
\sum_{i=0}^{\infty} \prod_{j=0}^{i}\E{\frac{F(j)}{F(j) + \alpha}} = \sum_{i=1}^{\infty} \E{e^{-\alpha \sum_{j=1}^{i} X_{j}}} < 1.
\end{equation}
Now, by Equation~\eqref{eq:det-lower-bound}, for any $\eta > 0$ such that $\eta \leq \inf_{i \geq k} x_{i}$, again by using properties of the exponential distribution, we have 
\begin{linenomath}
\begin{align} \label{eq:overtake-exp}
\prod_{i=k+1}^{\infty}\E{\exp{\left(\eta \left(X'_{i} - X_{i}\right) \right)}} &\leq  \prod_{i=k}^{\infty} \frac{x_{i}^{2}}{x_{i}^2 - {\eta}^2} 
=  \prod_{i=k}^{\infty}\left(1 +  \frac{{\eta}^{2}}{x_{i}^2 - {\eta}^2}\right) \leq \exp{\left(\sum_{i=k}^{\infty} \frac{{\eta}^{2}}{x_{i}^2 - {\eta}^2}\right)}.  
\end{align}
\end{linenomath}
If we choose $\eta = \eta_{k} := \sqrt{\left(2\sum_{i=k}^{\infty} \frac{1}{x_{i}^2} \right)^{-1}}$, note that, for each $i \geq k$, $\eta_{k} \leq x_{i}/\sqrt{2}$, and we readily verify that~\eqref{eq:overtake-exp} is bounded by $e^2$. One can make this choice whenever $\sum_{i=1}^{\infty} \frac{1}{x_{i}^2} < \infty$, and, moreover, since $\lim_{k \to \infty} \sum_{i=k}^{\infty} \frac{1}{x_{i}^2} = 0$, we have $\lim_{k \to \infty} \eta_{k} = \infty$. Thus, with $\alpha > 0$ as defined in~\eqref{eq:alpha-candidate}, if we choose $K := \inf{\left\{k \in \mathbb{N}\colon \eta_{k} \geq \alpha \right\}}$, Equation~\eqref{eq:persistent} of Theorem~\ref{thm:persistence} is satisfied. This proves the existence of a persistent hub. 

For uniqueness, we use Item~\ref{item:conv-series-and-borel} of Theorem~\ref{thm:unique}. The first condition is readily satisfied, since, by absolute continuity of the exponential distribution, clearly for any $j \in \mathbb{N}$, we have $\Prob{X_{j} \neq X'_{j}} = 1$. For the second condition, by using the inequality $e^{-x} \leq \frac{1}{x^{2}}$ and the definition of the exponential distribution, for any $\eps > 0$ we have
\begin{equation} \label{eq:borel-2}
\sum_{j=1}^{\infty} \Prob{X_{j} > \eps} = \sum_{j=0}^{\infty} \E{e^{-\eps F(j)}} \leq \frac{1}{\eps} \sum_{j=0}^{\infty} \frac{1}{x_{j}^2} < \infty. 
\end{equation}
The result follows. 
\end{proof}

\subsubsection{The proof of Theorem~\ref{thm:superlinear}}
This proof relies more heavily on the discrete dynamics of the generalised preferential tree. We first show that, the maximal out-degree of the process grows sufficiently quickly. 
In this regard, define
\[
\mathcal{M}_{n} := \max_{u \in \mathcal{T}_{n}} \outdeg{(u, \mathcal{T}_{n})}. 
\]
Recall the definition of $(\mathcal{Z}_{i})_{i \in \mathbb{N}_{0}}$ from~\eqref{eq:discrete-dynamics}.

First note that, if $N_{k}(n)$ denotes the number of nodes of out-degree $k$ in $\mathcal{T}_{n}$, by the handshaking lemma, we have 
\[\sum_{k=0}^{\mathcal{M}_{n}}(k+1)N_{k}(n) = 2n.\]
Therefore, under the condition~\eqref{eq:controlled-superlinear}, for all $n \in \mathbb{N}$, we have the deterministic bound
\begin{equation} \label{eq:det-parti-bounds}
\mathcal{Z}_{n} = \sum_{k=0}^{\mathcal{M}_{n}} f(k) N_{k}(n) \leq \frac{\kappa f(\mathcal{M}_{n})}{\mathcal{M}_{n} + 1} \sum_{k=0}^{\mathcal{M}_{n}} (k+1) N_{k}(n) = \frac{2n\kappa f(\mathcal{M}_{n})}{\mathcal{M}_{n}+1}.
\end{equation}

We then have the following lemma, closely related to Proposition 12 of Galashin:
\begin{lem} \label{lem:max-growth}
For any $\eps > 0$, there exists $r > 0$ such that
\begin{equation} \label{eq:max-growth}
    \Prob{\forall n \in \mathbb{N}\colon \mathcal{M}_{n}n^{-1/(2\kappa)} \geq r} = \Prob{\forall n \in \mathbb{N}\colon \mathcal{M}_{n} \geq \left\lceil n^{1/(2\kappa)}r\right\rceil} \geq 1 - \eps. 
\end{equation}
\end{lem}
\begin{rmq}
Lemma~\ref{lem:max-growth} uses a similar argument to \cite[Proposition~12]{galashin}. {\small\ensymboldremark }
\end{rmq}

\begin{proof}[Proof of Lemma~\ref{lem:max-growth}]
First, we prove the following claim:    
\begin{clm} \label{clm:super-martingale}
The sequence defined by $\mathcal{S}_1 := \frac{1}{\mathcal{M}_{1}} = 1$, and 
\[
\mathcal{S}_{n} := \frac{1}{\mathcal{M}_{n}}\prod_{i=1}^{n-1} \left(1-\frac{1}{2i\kappa}\right)^{-1}, \quad n \geq 2,
\]
is a super-martingale with respect to the filtration generated by $(\mathcal{T}_{i})_{i \in \mathbb{N}}$.
\end{clm}
Lemma~\ref{lem:max-growth} follows as a result of the following facts. We first note that, 
\begin{linenomath*}
    \begin{align*}
        \prod_{i=1}^{n-1} \left(1-\frac{1}{2i\kappa}\right)^{-1} = \prod_{i=1}^{n-1} \frac{2i \kappa}{2i\kappa - 1} = \prod_{i=1}^{n-1} \frac{i}{i - 1/(2\kappa)} = \frac{\Gamma(n) \Gamma(1- 1/(2\kappa))} {\Gamma(n - 1/(2\kappa))}. 
    \end{align*}
\end{linenomath*}
By using the fact that, for each $n \in \mathbb{N}$, $\alpha > 0$
\[n^{-\alpha} \Gamma(n)/\Gamma(n- \alpha) > 0 \quad \text{and} \quad \lim_{n \to \infty} n^{-\alpha} \Gamma(n)/\Gamma(n- \alpha) = 1,\] 
we may define \[c_{\min} := \inf_{n \in \mathbb{N}} \frac{n^{-1/(2 \kappa)} \Gamma(n) \Gamma(1- 1/(2\kappa)}{\Gamma(n- 1/(2\kappa))} > 0.\] But then, by Doob's martingale inequality, 
\begin{linenomath*}
    \begin{align*}
        \Prob{ \exists j \colon j^{-1/2\kappa} \mathcal{M}_{j} < r} = \Prob{\sup_{j \in \mathbb{N} } j^{1/2\kappa} \mathcal{M}^{-1}_{j} > \frac{1}{r}} \leq \Prob{\sup_{j \in \mathbb{N}} \mathcal{S}_{j} > \frac{c_{\min}}{r} } \leq \E{\mathcal{S}_{1}} r/c_{\min} = r/c_{\min}.
    \end{align*}
\end{linenomath*}
By setting $r = \eps c_{\min}$ we deduce Equation~\eqref{eq:max-growth}. 
\end{proof}

\begin{proof}[Proof of Claim~\ref{clm:super-martingale}]
Note that, since there may be multiple nodes in $\mathcal{T}_{n}$ with out-degree $\mathcal{M}_{n}$, we have
\[
\E{1/\mathcal{M}_{n+1} \,| \, \mathcal{T}_{n}} = \begin{cases}
    \frac{1}{\mathcal{M}_{n} + 1} & \text{with probability at least } f(\mathcal{M}_{n})/\mathcal{Z}_{n},
    \\
    \frac{1}{\mathcal{M}_{n}} & \text{with probability at most } 1 - f(\mathcal{M}_{n})/\mathcal{Z}_{n}. 
\end{cases}
\]
Therefore, for $n \in \mathbb{N}$
\begin{linenomath*}
\begin{align*}
\E{1/\mathcal{M}_{n+1} \,| \, \mathcal{T}_{n}} & \leq \frac{1}{\mathcal{M}_{n}}\left(1 - \frac{f(\mathcal{M}_{n})}{\mathcal{Z}_{n}}\right) + \frac{1}{\mathcal{M}_{n} +1} \frac{f(\mathcal{M}_{n})}{\mathcal{Z}_{n}} = \frac{1}{\mathcal{M}_{n}} - \frac{1}{\mathcal{M}_{n}(\mathcal{M}_{n} + 1)}\frac{f(\mathcal{M}_{n})}{\mathcal{Z}_{n}}
\\ & = \frac{1}{\mathcal{M}_{n}}\left(1 - \frac{f(\mathcal{M}_{n})}{\mathcal{Z}_{n} (\mathcal{M}_{n} + 1)}\right) \stackrel{\eqref{eq:det-parti-bounds}}{\leq} \frac{1}{\mathcal{M}_{n}}\left(1 - \frac{1}{2n\kappa}\right). 
\end{align*}
\end{linenomath*}
Multiplying both sides by $\prod_{i=1}^{n} \left(1-\frac{1}{2i\kappa}\right)^{-1}$, we conclude the proof.
\end{proof}

The proof of Theorem~\ref{thm:superlinear} works by using a Borel-Cantelli argument to show that the event
    \[
    \left\{\text{the $n$th node catches up to a node of maximal out-degree in $\mathcal{T}_{n}$}\right\}, 
    \]
    occurs only finitely often, which shows that Equation~\eqref{eq:catch-up-set} from Lemma~\ref{lem:finite-catch-up} is satisfied. We wish to define this event more formally, with respect to the underlying process $(\mathscr{T}_{t})_{t \geq 0}$. In this regard, for $n \in \mathbb{N}$, 
  \begin{equation} \label{eq:max-n}
    m_{n} := \min\left\{u \in \mathscr{T}_{\tau_{n}}: \outdeg{u, \mathscr{T}_{\tau_{n}}} = M_{n}\right\} \quad \text{and} \quad o_{n} := \min\left\{u \in \mathscr{T}_{\tau_{n}} \setminus \mathscr{T}_{\tau_{n-1}}\right\},  
    \end{equation}
where, in both cases, the minimum refers to the lexicographical ordering on $\mathcal{U}$. Therefore, $m_{n}$ denotes the Ulam-Harris label of a node that attains the maximal out-degree in $\mathscr{T}_{\tau_{n}}$, and $o_{n}$ denotes the Ulam-Harris label of the $n$th node in the process $(\mathcal{T}_{n})_{n \in \mathbb{N}}$.
\begin{proof}[Proof of Theorem~\ref{thm:superlinear}]
Note that, as defined in~\eqref{eq:max-n}, the values $o_{n}$ and $m_{n}$ are measurable with respect to $\mathcal{F}_{\tau_{n}}$. Additionally, by definition, $\outdeg{m_{n}, \mathscr{T}_{\tau_{n}}} = \mathcal{M}_{n}$. Thus, by the memoryless property of the exponential distribution, conditional on $\mathcal{F}_{\tau_{n}}$, the waiting time 
\[
X'(m_{n}(\mathcal{M}_{n}+1)) := X(m_{n}(\mathcal{M}_{n}+1)) - \tau_{n} \sim \Exp{f(\mathcal{M}_{n})}. 
\]
Note also that the random variables $X'(m_{n}(\mathcal{M}_{n}+1)), (X(m_{n}(\mathcal{M}_{n}+j)))_{j \geq 2}$ and $(X(o_{n}j))_{j \in \mathbb{N}}$ are independent of $\mathcal{F}_{\tau_{n}}$. 
Consequentially, conditional on $\mathcal{F}_{\tau_{n}}$, we have
\begin{linenomath}
\begin{align*}
&\mathcal{A}_{n} := \left\{\exists t > 0\colon \outdeg{o_{n}, \mathscr{T}_{t}} \geq \outdeg{m_{n}, \mathscr{T}_{t}} \right\} \\ & \hspace{3cm} = \left\{\exists j \colon \sum_{i=1}^{\mathcal{M}_{n} + j} X(o_{n}i) \leq X'(m_{n}(\mathcal{M}_{n}+1)) + \sum_{i = \mathcal{M}_{n} + 2}^{\mathcal{M}_{n} + j} X(m_{n}i)\right\}.
\end{align*}
\end{linenomath}
The event on the right hand side says that, for some $j$, $o_{n}$ reaches out-degree $\mathcal{M}_{n} + j$ before $m_{n}$ does. Now, by Lemma~\ref{lem:overtake-bound}, with $(X_{i})_{i \in \mathbb{N}}$, $(X_{i})_{i \in \mathbb{N}}$ independent sequences with $X_{i} \sim \Exp{f(i-1)}$, for $\lambda > 0$ we have 
\begin{equation} \label{eq:moment-bound-new-catch-up}
\Prob{\mathcal{A}_{n} \, | \, \mathcal{F}_{\tau_{n}}} \leq \left(\prod_{i=\mathcal{M}_{n} +1}^{\infty}\E{e^{\lambda \left(X_{i}' - X_{i}\right)} \, \bigg | \, \mathcal{M}_{n}}\right) \E{e^{-\lambda \sum_{k=1}^{\mathcal{M}_{n}} X_{k}}\, \bigg | \, \mathcal{M}_{n}}.   
\end{equation}
Now, for each $r > 0$ we also define 
\begin{equation} \label{eq:good-concentration}
   b^{r}_{n} := \left\lceil n^{1/(2\kappa)}/r\right\rceil \quad \text{and} \quad  \mathcal{B}^{r} := \left\{\forall n \in \mathbb{N}\colon \mathcal{M}_{n} \geq b^{r}_{n}\right\}. 
\end{equation}
Noting that the right-side of~\eqref{eq:moment-bound-new-catch-up} is decreasing in the value of $\mathcal{M}_{n}$, by combining~\eqref{eq:moment-bound-new-catch-up} and~\eqref{eq:good-concentration}, for any $\lambda < f(b^{r}_{n})$ 
\begin{linenomath}
\begin{align} \label{eq:event-bound}
\nonumber & \Prob{\mathcal{A}_{n} \cap \mathcal{B}^{r}} \leq \E{e^{\lambda\sum_{i = b^{r}_{n} + 1}^{\infty} \left(X_{i}' - X_{i}\right)}} \E{e^{-\lambda \sum_{k=1}^{b^{r}_{n}} X_{k}}}
= \prod_{i=b^{r}_{n}}^{\infty} \left(1 + \frac{\lambda}{f(i)^2 - \lambda}\right) \prod_{k=0}^{b^{r}_{n} -1} \left(1+ \frac{\lambda}{f(k) + \lambda}\right)
\\ & \hspace{1cm} \leq \exp\left(\sum_{i=b^{r}_{n}}^{\infty}\frac{\lambda}{f(i)^{2} - \lambda}\right) \exp\left(-\sum_{k=0}^{b^{r}_{n} - 1} \frac{\lambda}{f(k) + \lambda}\right) \leq \exp\left(\sum_{i=b^{r}_{n}}^{\infty}\frac{\lambda}{f(i)^{2} - \lambda}\right) \exp\left(-\frac{\lambda}{f(0)}\right). 
\end{align}
\end{linenomath}
Note that~\eqref{eq:controlled-superlinear} implies that for each $\ell \in \mathbb{N}$ 
\begin{equation}
f(\ell) \geq (f(0)/\kappa) (\ell+1), 
\end{equation}
 and thus, by a (crude) integral test bound, that for $k \in \mathbb{N}$, 
 \begin{equation} \label{eq:crude-bound}
 \sum_{\ell=j}^{\infty} 1/f(\ell)^2 \leq (\kappa/f(0))^2 \sum_{\ell = j}^{\infty} \frac{1}{(\ell + 1)^2} < \frac{2(\kappa/f(0))^2}{j}. 
 \end{equation}
  In a similar manner to the passage following~\eqref{eq:overtake-exp}, choose 
\begin{equation} \label{eq:lambda-bound}
\lambda = \lambda_{n} := \sqrt{\left(2\sum_{i=b^{r}_{n}}^{\infty} \frac{1}{f(i)^2} \right)^{-1}} \stackrel{\eqref{eq:crude-bound}}{>}  \frac{\sqrt{b^{r}_{n}}f(0)}{2\kappa} \stackrel{\eqref{eq:good-concentration}}{\geq} \frac{f(0) n^{1/(4\kappa)}}{2\kappa \sqrt{r}}. 
\end{equation}
Then, since $\lambda_{n} \leq f(j)/2$ for each $j\geq b^{r}_{n}$, we can bound the 
right-side of~\eqref{eq:event-bound} so that  
\begin{equation} \label{eq:upper-bound-borel-cantelli}
\Prob{\mathcal{A}_{n} \cap \mathcal{B}^{r}} \stackrel{\eqref{eq:lambda-bound}}{<} e^{2} \exp\left(-\frac{n^{1/(4\kappa)}}{2\kappa \sqrt{r}}\right).  
\end{equation}
The right-side of~\eqref{eq:upper-bound-borel-cantelli} is summable, hence, by the Borel-Cantelli lemma, we have 
\begin{equation}
\Prob{\mathcal{A}_{n} \cap \mathcal{B}^{r} \text{ occurs for infinitely many $n$ }} = 0. 
\end{equation}
By Lemma~\ref{lem:max-growth}, for any $\eps > 0$, there exists $r > 0$ such that $\Prob{\mathcal{B}^{r}} \geq 1 - \eps$, hence, for any $\eps > 0$, we have 
\begin{equation}
\Prob{\mathcal{A}_{n} \text{ occurs for infinitely many $n$ }} < \eps. 
\end{equation}
Hence, with probability one, $\mathcal{A}_{n}$ only occurs finitely often. This implies that the set 
\[
\left\{o_{n}: o_{n} \text{ catches up to a node of maximal out-degree}\right\}
\]
is finite, hence there is a finite set $\mathcal{P}$ which contains nodes of maximal out-degree, for all but finitely many $n$. A similar argument to the proof of Item~\ref{item:persistent-hub} of Theorem~\ref{thm:persistence}, using~\eqref{eq:finite-leadership}, now implies the existence of a persistent hub. Uniqueness follows from a similar argument to the proof of Theorem~\ref{thm:main}, using the deterministic values $f(j)$ instead of $x_{j}$ in~\eqref{eq:borel-2}.
\end{proof}

\section*{Acknowledgements}
The author is funded by Deutsche Forschungsgemeinschaft (DFG) through DFG Project no. $443759178$. 

\bibliographystyle{abbrv}
\bibliography{refs}

\end{document}